%% file: tarpodual.tex
\documentclass[10pt,a4paper,DIV=classic]{myart}
\input{head2013}

\usepackage{marginnote}
\usepackage{todonotes}

\begin{document}

\title{Dual Representation of Minimal Supersolutions of Convex BSDEs\tnoteref{t}}
\tnotetext[t]{We thank the anonymous referees for valuable comments and suggestions.}

\keyAMSClassification{60H10, 01A10, 60H20}
\ArXiV{1308.1275}

\author[a,1,s]{Samuel Drapeau}
\author[b,2,s]{Michael Kupper}
\author[c,3]{Emanuela Rosazza Gianin}
\author[a,4,u]{Ludovic Tangpi}
\address[a]{Humboldt-Universit\"at Berlin, Unter den Linden 6, 10099 Berlin, Germany}
\address[b]{Universit\"at Konstanz, Universit\"atsstra\ss e 10, 78464 Konstanz, Germany}
\address[c]{University of Milano-Bicocca, Via Bicocca degli Arcimboldi 8, 20126 Milan, Italy}

\eMail[1]{drapeau@math.hu-berlin.de}
\eMail[2]{kupper@uni-konstanz.de}
\eMail[3]{emanuela.rosazza1@unimib.it}
\eMail[4]{tangpi@math.hu-berlin.de}


\myThanks[s]{MATHEON project E.11}
\myThanks[u]{Funding: Berlin Mathematical School, phase II scholarship}

\date{\today}

\abstract{
	 We give a dual representation of minimal supersolutions of BSDEs with non-bounded, but integrable terminal conditions and under weak requirements on the generator which is allowed to depend on the value process of the equation. 
	 Conversely, we show that any dynamic risk measure satisfying such a dual representation stems from a BSDE.
	 We also give a condition under which a supersolution of a BSDE is even a solution.
}
\keyWords{Convex Duality; Supersolutions of BSDEs; Cash-Subadditive Risk Measures.}

\maketitle

\input{introduction}
\input{chap01}
\input{chap02}
\input{chap03}
\input{chap04}
\input{chap05}
\begin{appendix}
\input{appendix01}
\end{appendix}
\bibliographystyle{abbrvnat}
\bibliography{bibliography}

\end{document}

%% file: head2013.tex
\usepackage[normalem]{ulem}

\newcommand{\norm}[1]{\left\Vert#1\right\Vert}
\newcommand{\abs}[1]{\left\vert#1\right\vert}
\newcommand{\Set}[1]{\ensuremath{ \left\{ #1 \right\} }}
\newcommand{\set}[1]{\ensuremath{ \{ #1 \} }}
\newcommand{\R}{\mathbb{R}}

\renewcommand{\mid}{\,|\,}
\newcommand{\Mid}{\:\big | \:}
\newcommand{\mmid}{\;\Big | \;}

\newcommand*{\cadlag}{c\`adl\`ag}

\DeclareMathOperator*{\esssup}{ess\,sup}
\DeclareMathOperator*{\essinf}{ess\,inf}

%% file: introduction.tex
\section{Introduction}

Since their introduction by \citet{peng01}, nonlinear Backward Stochastic Differential Equations (BSDEs) have found numerous applications in mathematical finance.
For instance, they are used to constructively describe the optimal solution of some utility maximization problems, see \citet{HIM05}.
Through the $g$-expectations of \citet{peng03}, BSDEs offer a framework to study nonlinear expectations and time consistent dynamic risk measures as described by \citet{gianin02} and \citet{Del-Peng-Ros}.
Mainly driven by its financial applications, the study of BSDEs has been extended in various ways beyond the question of existence and uniqueness of solutions. 
Many authors have been interested in questions such as numerical approximation, structural and path properties of BSDE solutions, see for instance the survey of \citet{karoui01} for an overview.
The subject of this paper is to study BSDEs by convex duality theory.

Deviating from the usual quadratic growth or Lipschitz assumptions on the generator of the BSDE, \citet{DHK1101} show existence of the minimal supersolution of a BSDE. 
They study the properties of minimal supersolutions and give the link to cash-subadditive risk measures of \citet{ravanelli01}.
Our main objectives are, on the one hand, to derive a dual representation of minimal supersolutions of BSDEs, and, on the other hand, to study conditions under which an operator satisfying such a representation is the minimal supersolution or a solution of a BSDE.

Dual representation of solutions of BSDE with quadratic growth in the control variable, linear growth in the value process and bounded terminal condition are by now well understood, see for instance \citet{barrieu01} and \citet{ravanelli01}.

In this work we give the dual representation of the minimal supersolution functional of a BSDE in the framework of \citet{DHK1101}.  
The $\mathcal{H}^1$-$L^\infty$ duality turns out to be the right candidate to constitute the basis of our representation.
As a starting point, we consider the set of essentially bounded terminal conditions. In this case, we obtain a dual representation of the minimal supersolution at time $0$ and a pointwise robust representation in the dynamic case.
We show that when the generator of the equation is decreasing in the value process, the minimal supersolution defines a time consistent cash-subadditive risk measure.
It allows for a dual representation on the space of essentially bounded random variables, which agrees with the representation of \citet{ravanelli01} obtained for BSDE solutions.
Our dual representation is obtained by showing that the representation of \citet{ravanelli01} can be restricted on a smaller set.
Then we can use truncation and approximation arguments to obtain the representation in the general case, due to monotone stability of minimal supersolutions. 
A direct consequence of our representation is the identification of BSDEs solution and minimal supersolution in the case of linear growth generators.
Note that our truncation technique appears already in the work of \citet{Del-Peng-Ros} where it is used to construct a sequence of $\mu$-dominated risk measures. 
Furthermore, prior to us \citet{barrieu01} and \citet{Bion} already used the $BMO$-martingale theory in the study of financial risk measures, but in different settings from ours.
Using standard convex duality arguments such as the Fenchel-Moreau theorem and the properties of the Fenchel-Legendre transform of a convex functional, we extend our dual representation to the set of random variables that can be identified to $\mathcal{H}^1$-martingales. 
Notice that this representation is obtained in the static case.

Our representation results can be seen as extensions of the dual representation of the minimal super-replicating cost of  \citet{EQ95} to the case where we allow for a nonlinear cost function in the dynamics of the wealth process.

The second theme of this work is to give conditions based on convex duality under which a dynamic cash-subadditive risk measure with a given representation can be seen as the solution, or the minimal supersolution of a BSDE. 
The cash-additive case has been studied by \citet{BDH10}. 
Their results are based on $m$-stability of the dual space, some supermartingale property and Dood-Meyer decomposition of the risk measure. 
We shall show that in the cash-subadditive case, discounting the risk measure yields similar results, hence showing an equivalent relationship between existence of the minimal supersolution and the dual representation.

The rest of the paper is structured as follows:
The next section is dedicated to the setting of the probabilistic framework of our study. We also introduce the notation and gather some results on minimal supersolution of BSDEs. Our representation results are stated and proved in Section \ref{Sec:representation}. The question of deriving a BSDE from the representation is dealt with in the last section.

%% file: chap01.tex
\section{Minimal Supersolution of Convex BSDEs}
Given a fixed time horizon $T>0$, let $(\Omega,\mathcal{F},(\mathcal{F}_t)_{t \in [0,T]},P)$ be a filtrated probability space.
We assume that the filtration $(\mathcal{F}_t)$ is generated by a $d$-dimensional Brownian motion $W$ and it satisfies the usual conditions.
We further assume that $\mathcal{F}_T=\mathcal{F}$.
The set of $\mathcal{F}_t$ measurable random variables is denoted by $L^0_t$ where random variables are identified in the $P$-almost sure sense.
For $1\leq p < \infty$, we denote by $L^p_t$ the set of random variables in $L^0_t$ which are $p$-integrable and set $L^p = L^p_T$, and $L^\infty$ is the set of essentially bounded random variables in $L^0_T$.
Statements concerning random variables or processes like inequalities and equalities are to be understood in the $P$-almost sure or $P\otimes dt$-almost sure sense, respectively.
The set of stopping times with values in $[0,T] $ is denoted by $\mathcal{T}$.
We consider the sets of processes 
\begin{align*}
	\mathcal{S}   &:=\Set{Y:\Omega\times [0,T] \to \mathbb{R} ; Y \text{ is adapted and \cadlag}};\\
	\mathcal{L}   &:=\Set{Z: \Omega\times [0,T] \to \mathbb{R}^d; Z\text{ is predictable, and }\int_{0}^{T}\norm{ Z_s }^2ds <+\infty};
\end{align*}
\begin{align*}
	\mathcal{H}^p &:=\Set{X \in \mathcal{S} : X\text{ is a continuous martingale with }\sup_{t \in [0,T]}\abs{X_t}\in L^p};\\
	BMO      &:=\Set{M:M\in \mathcal{H}^1 \text{ such that } \norm{ M }_{BMO} < \infty},
\end{align*}
where $\norm{ M }_{BMO} := \sup_{\tau\in \mathcal{T}}\Vert E\left[ \langle M \rangle_T - \langle M \rangle_{\tau}\mid \mathcal{F}_{\tau} \right]^{\frac{1}{2}} \Vert_{\infty}$.
The set $\mathcal{H}^1_+$ denotes the set of non-negative martingales in $\mathcal{H}^1$.
Further, let $L^\infty_+$ and $L^\infty_{++}$ be the sets of non-negative and strictly positive random variables in $L^\infty$, respectively.
Notice that $X_t=E[X_T\mid \mathcal{F}_t]$ for all $0\leq t\leq T$ and every $X \in \mathcal{H}^1$.
Therefore, $\mathcal{H}^1$  will be identified with the set of random variables $X \in L^1$, satisfying $\sup_{t\in [0,T]}\abs{E[X \mid \mathcal{F}_t]}\in L^1$.
The dual of the Banach space $\mathcal{H}^1$ can be identified with $BMO$, see \citet[Theorem 2.6]{Kazamaki01}.

We further consider the sets
\begin{align*}
	 \mathcal{Q}&:=\Set{q \in \mathcal{L}:\exp\left( \int_{0}^{T}q_u dW_u -\frac{1}{2}\int_{0}^{T}\norm{q_u}^2 du \right)\in L^\infty},\\
	\mathcal{D} &:=\Set{\beta: \Omega\times [0,T] \to \mathbb{R}; \beta \text{ predictable, $\int_0^T\beta^-_udu \in L^\infty$ and } \int_0^T\beta_u^+\,du < \infty }.
\end{align*}
In our setting, the dual variables will appear to be closely related to the sets $\mathcal{D} $ and $\mathcal{Q} $.
The idea of defining the set $\mathcal{Q}$ with stochastic exponentials in $L^\infty$ is motivated by the fact that the representation will rely on the $\mathcal{H}^1$-$L^\infty$ duality.
For $q \in \mathcal{Q}$, we denote by $Q^q$ the probability measure whose density process is given by the stochastic exponential $M^q := \exp( \int_{}^{}q_u dW_u -\frac{1}{2}\int_{}^{}q_u^2 du )$ and for $\beta \in \mathcal{D}$ we denote by $D^\beta_{s,t}:=\exp(-\int_s^t \beta_u du)$, $0\leq s\leq t\leq T$ the discounting factors with respect to $\beta$.
In the case where $\beta \in \mathcal{D}_+ := \set{\beta \in \mathcal{D}: \beta \ge 0} $, the measures with density $M^q_tD^\beta_{0,t}$ was referred to by \citet{ravanelli01} as subprobability measures.

A generator is a jointly measurable function $g:\Omega\times [0,T]\times \mathbb{R} \times \R^d\to (-\infty,+\infty]$ where $\Omega \times [0,T]$ is endowed with the predictable $\sigma$-field, and such that $(y,z)\mapsto g_t(\omega,y,z)$ is $P\otimes dt$-almost surely lower semicontinuous.
We denote by $g^\ast$ the pointwise Fenchel-Legendre transform of $g$, that is
\begin{equation*}
	g^\ast_t(\omega, \beta,q)=\sup_{(y,z)\in \mathbb{R}\times \R^d}\Set{-y\beta+qz-g_t(\omega,y,z)}, \quad (\beta,q) \in \R\times \R^d,
	\label{}
\end{equation*}
where the scalar product between two vectors $q, z \in \mathbb{R}^d$ is denoted by $qz: = q\cdot z$. 
For any $(\beta,q) \in \mathbb{R}\times \mathbb{R}^d$, the process $g^\ast(\beta, q)$ is predictable, see \citet[Proposition 14.40]{rockafellar02}.

%% file: chap02.tex

Following \citet{DHK1101}, a supersolution of the BSDE with terminal condition $X \in L^0$ and driver $g$ is defined as a couple $(Y,Z) \in \mathcal{S}\times \mathcal{L}$ such that 
\begin{equation}
	\begin{cases}
		\displaystyle Y_s-\int_{s}^{t}g_u(Y_u,Z_u)du+\int_{s}^{t}Z_u dW_u\geq Y_t, \quad \text{for every} \quad 0\leq s\leq t\leq T\\
		\displaystyle Y_T\geq X.
	\end{cases}
	\label{eq:supersolutions}
\end{equation}
The following equivalent formulation of \eqref{eq:supersolutions} will sometimes be useful: a pair $(Y,Z)$ is a supersolution if and only if there exists a c\`adl\`ag, increasing and adapted process $K$ with $K_0 =0$ such that
\begin{equation}
	Y_t= X + \int_{t}^{T}g_u(Y_u,Z_u)du +  (K_T - K_t)-\int_{t}^{T}Z_u dW_u, \quad \text{for every}\quad 0\leq t\leq T.
	\label{eq:equiv_supersol}
\end{equation}
The control process $Z$ of a supersolution $(Y,Z)$ is said to be admissible if the continuous local martingale $\int Z\,dW$ is a supermartingale.
Given a driver $g$ we define 
\begin{equation*}
	\mathcal{A}\left(X\right):=\Set{(Y,Z)\in \mathcal{S}\times \mathcal{L}:(Y,Z)\text{ fulfills \eqref{eq:supersolutions} and $Z$ is admissible}}, \quad X \in L^0.
	\label{eq:set:supersolutions}
\end{equation*}
A supersolution $(\bar{Y},\bar{Z})\in \mathcal{A}(X)$ is said to be minimal if $\bar{Y}\leq Y$ for every $(Y,Z)\in \mathcal{A}(X)$.
A generator $g$ is said to be 
\begin{enumerate}[label=(\textsc{Pos}),leftmargin=40pt]
	\item positive, if $g\geq 0$;\label{pos}
\end{enumerate}
\begin{enumerate}[label=(\textsc{Dec}),leftmargin=40pt]
	\item decreasing, if $g(y,z)\leq g(y^\prime,z)$ whenever $y\geq y^\prime$;\label{dec}
\end{enumerate}
\begin{enumerate}[label=(\textsc{Conv}),leftmargin=40pt]
	\item convex, if $(y,z)\mapsto g(y,z)$ is convex;\label{jconv}
\end{enumerate}
\begin{enumerate}[label=(\textsc{Lsc}),leftmargin=40pt]
	\item lower semicontinuous, if $(y,z)\mapsto g(y,z)$ is lower semicontinuous.\label{lsc}
\end{enumerate}
\begin{theorem}
\label{thm:existence}
	Let $g$ be a driver satisfying \ref{jconv}, \ref{lsc} and \ref{pos}.
	For any $X \in \mathcal{X}: = \set{X \in L^0: X^- \in L^1}$ such that $\mathcal{A}(X)\neq \emptyset$, there exists a unique minimal supersolution $(\bar{Y},\bar{Z})\in \mathcal{A}(X)$ which satisfies
	\begin{equation*}
		\bar{Y}_t = \essinf\Set{ Y_t: (Y,Z) \in \mathcal{A}(X)}\quad \text{for all } t \in [0,T].
	\end{equation*}
\end{theorem}
\begin{proof}
	See Appendix \ref{appendix01}.
\end{proof}
For a generator $g$ which satisfies \ref{jconv}, \ref{lsc} and \ref{pos} we define the operator $\mathcal{E}: \mathcal{X} \to \mathcal{S} \cup \set{\infty} $ as
\begin{align*}
	\mathcal{E} :  X \longmapsto \begin{cases}
		\bar{Y} &\text{ if } \mathcal{A}(X)\neq \emptyset\\
		+\infty & \text{ else},
	\end{cases}
\end{align*}
where $\bar{Y}$ is defined in Theorem \ref{thm:existence} and depends on $X$.
 We conclude this section by the following structural properties and stability results for $\mathcal{E}$. 
\begin{proposition}
	\label{thm:properties}
	Let $g$ satisfying \ref{jconv}, \ref{lsc} and \ref{pos},  let $X, X' \in L^0$ and $m \in \mathbb{R}$. It holds
	\begin{enumerate}[label=\textit{(\roman*)}]
		\item \emph{Monotonicity:} if $X' \leq X$ then $\mathcal{E}(X') \leq \mathcal{E}(X)$;
		\item \emph{Convexity:} $\mathcal{E}_0(\lambda X + (1-\lambda)X') \leq \lambda\mathcal{E}_0(X) + (1-\lambda)\mathcal{E}_0(X')$, for all $\lambda \in (0,1)$;
		\item\label{point3} \emph{Cash-subadditivity:} if  $g$ is \ref{dec} and $m\geq 0$,
			then $\mathcal{E}_0(X + m) \leq \mathcal{E}_0(X) + m$;
		\item\label{point4} \emph{Cash-additivity:} if  $g:(y,z) \mapsto g(z)$, 
			then: $\mathcal{E}_0(X + m) = \mathcal{E}_0(X) + m$;
		\item\label{point5} \emph{Normalization:} for every $y \in \mathbb{R} $ such that  $g(y,0) = 0$ it holds $\mathcal{E}_0(y) = y$.
	\end{enumerate}
	Furthermore, for any sequence of random variables $(X_n)\subseteq L^0$ such that $\inf_n X_n \in L^1$, it holds
	\begin{enumerate}[resume*]
		\item \emph{Monotone convergence:} $\lim \mathcal{E}_0(X_n)=\mathcal{E}_0(X)$ whenever $(X_n)$ is increasing and converges $P$-a.s. to $X \in L^0$;
		\item \emph{Fatou:} $\mathcal{E}_0(\liminf X_n)\leq \liminf\mathcal{E}_0(X_n)$.
	\end{enumerate}
	As a restriction on $L^1$ the operator $\mathcal{E}_0$ is $L^1$-lower semicontinuous.
\end{proposition}
\begin{proof}
	See Appendix \ref{appendix01}.
\end{proof}

%% file: chap03.tex
\section{Dual Representation}\label{Sec:representation}
\subsection{The Bounded Case} 
The following proposition provides the dual representation of $g$-expectations, see also \citep[Proposition 3.3]{karoui01}.
Note that such a representation was already obtained in \citep{ravanelli01} in the more general quadratic case, where the value function of the BSDE was written as a supremum over a set of measures with uniformly integrable densities.
Here, we show that under the linear growth assumption the representation can be restricted to a set of measures with densities in $L^\infty$.
\begin{proposition}
	\label{g_expectation}
	Let $X \in L^{\infty}$ and $f$ be a driver satisfying \ref{jconv}, \ref{lsc} and \ref{pos}, as well as the linear growth condition
	\begin{equation*}
		 f(y,z) \leq a + b \abs{ y } + c \norm{ z }, \quad a,b,c >0.
	\end{equation*}
	Then the solution $(Y,Z)$ of the BSDE
	\begin{equation}
		\label{eq:BSDE-f-bounded}
		Y_t = X + \int_t^Tf_u(Y_u,Z_u)\,du - \int_t^TZ_u\,dW_u, \quad t\in [0,T],
	\end{equation}
	admits the dual representation
	\begin{equation}
		Y_t = \esssup_{(\beta,q) \in \mathcal{D}\times \mathcal{Q}}E_{Q^q}\left[ D_{t,T}^\beta X - \int_t^T D_{t,u}^\beta f^\ast_u(\beta_u,q_u)\,du \mmid \mathcal{F}_t \right], \quad t\in [0,T].
	\end{equation}
\end{proposition}
Before going through the proof, let us provide the following well known lemma, see \citep{ravanelli01}.
\begin{lemma}
	\label{linear_growth}
	Let $f: \R \times \R^d \to (-\infty,\infty]$ be a function satisfying \ref{lsc}, \ref{jconv} as well as
	\begin{equation*}
		\abs{ f(y,z) } \leq a + b\abs{ y } + c\norm{ z }, \quad (y,z) \in \mathbb{R}\times\mathbb{R}^d
	\end{equation*}
	for some positive constants $a, b$ and $c$.
	Then, $f$ admits for all $(y,z)\in \R\times \R^d$ the dual representation 
	\begin{equation}\label{eq:blablabla}
		f(y,z) = \max_{\beta\in \R,q\in \R^d}\Set{ -\beta y + qz - f^\ast(\beta,q) }=-\bar{\beta} y+\bar{q}z-f^\ast(\bar{\beta},\bar{q})
	\end{equation}
	for some $\abs{ \bar{\beta} } \leq b$ and $\norm{ \bar{q} } \leq c$.
\end{lemma}
\begin{proof}
	We shortly present the argument.
	First, the dual representation of $f$ is a consequence of the Fenchel-Moreau theorem, since the growth condition implies that $f$ is proper.
	Second, the growth condition on $f$ implies $f^\ast(\beta,q)\geq -\beta y +qz-f(y,z)\geq -a -\beta y+qz -b\abs{y}-c\norm{z}$ for all $(y,z)\in \R \times \R^d$.
	In particular $f^\ast(\beta,q)\geq -a+m \abs{\beta}(\abs{\beta}-b)+n\norm{q}(\norm{q}-c)$ for every $n,m \in \mathbb{N}$, showing that $f^\ast(\beta,q)=\infty$ for all $b<\abs{\beta}$ or $c<\norm{q}$.
	Hence, the supremum in \eqref{eq:blablabla} can be restricted to $ \abs{\beta} \leq b$ and $ \norm{q}\leq c$.
	Finally, $f^\ast$ being lower semicontinuous and having a domain contained in a compact set, the supremum is therefore a maximum.
\end{proof}

\begin{proof} [Proof of Proposition \ref{g_expectation}]
	First notice that by Lemma \ref{linear_growth}, $f$ is globally Lipschitz, due to the boundedness of $\bar{q}$ and $\bar{\beta}$, which ensures existence and uniqueness of a strong solution for the BSDE with bounded terminal condition, see \citep{peng01}.
	Let $(\beta, q) \in \mathcal{D}\times\mathcal{Q} $.
	With the same arguments as in \citep{karoui01,ravanelli01}, using It\^o's formula applied to $D_{t,u}^\beta Y_u$ between $t$ and $T$ where $(Y,Z)$ is the solution of the Lipschitz BSDE with bounded terminal condition \eqref{eq:BSDE-f-bounded}, it holds
	\begin{align*}
		Y_t & = D_{t,T}^\beta X - \int_t^T D_{t,u}^\beta \left( -\beta_uY_u + q_uZ_u - f_u\left(Y_u,Z_u\right) \right)\,du - \int_t^T D_{t,u}^\beta Z_u\,dW_u^{Q^q}\\
		    & = E_{Q^q}\left[ D_{t,T}^\beta X - \int_t^T D_{t,u}^\beta \left( -\beta_uY_u + q_uZ_u - f_u\left(Y_u,Z_u\right) \right)\,du \mmid \mathcal{F}_t \right]
	\end{align*}
	for all $(\beta,q) \in \mathcal{D}\times \mathcal{Q}$.
	since $-\beta_uY_u + q_uZ_u - f_u(Y_u,Z_u)\leq f^\ast_u(\beta_u,q_u)$, it follows
		\begin{equation}
		\label{first_Y_ineq}
		Y_t\geq \esssup_{(\beta,q) \in \mathcal{D}\times \mathcal{Q}}E_{Q^q}\left[ D_{t,T}^\beta X - \int_t^T D_{t,u}^\beta f^\ast_u(\beta_u,q_u)\,du \mmid \mathcal{F}_t \right].
	\end{equation}
	For the other inequality, since $f$ satisfies the conditions of Lemma \ref{linear_growth}, for all $(\omega,t) \in \Omega \times [0,T]$ the subgradients $\partial f(\omega,t,Y_t,Z_t)$ with respect to $(Y_t, Z_t)$ are non-empty for all $(\omega,t) \in \Omega\times [0,T]$.
	Therefore, by means of \citep[Theorem 14.56]{rockafellar02}, we can apply a measurable selection theorem, see for instance \citep[Corollary 1C]{rockafellar03}, to assert the existence of a predictable $\mathbb{R}\times \mathbb{R}^d$-valued process $(\bar{\beta}, \bar{q})$ such that
	\begin{equation}
		\label{eq:f_attained}
		f(Y,Z) = -\bar{\beta}Y + \bar{q}Z - f^\ast(\bar{\beta}, \bar{q}),\quad P \otimes dt\text{-a.s.,}
	\end{equation}
	and $ \abs{\bar{\beta}} \leq b$ and $\norm{ \bar{q} } \leq c$.
	Hence,
	\begin{equation}
		\label{Y_q_bar}
		Y_t = E_{Q^{\bar{q}}}\left[ D_{t,T}^{\bar{\beta}} X - \int_t^T  D_{t,u}^{\bar{\beta}}f^\ast_u(\bar{\beta}_u,\bar{q}_u)\,du \mmid \mathcal{F}_t \right].
	\end{equation}
	But even though $\bar{q}$ is bounded it is not guaranteed that the density of $Q^{\bar{q}}$ belongs to $L^\infty$. 
	Thus, we introduce the following localization by defining
	\begin{equation*}
		\sigma^n := \inf\Set{ s >0: \abs{ \int_0^s\bar{q}_u\,dW_u } \geq n } \wedge T,\quad n\in \mathbb{N},
	\end{equation*}
	and put $\bar{q}^{n}:= \bar{q}1_{[0,\sigma^n]} \in \mathcal{Q}$ and $\bar{\beta}^{n} := \bar{\beta}1_{[0,\sigma^n]} \in \mathcal{D}$.
	Then, since $\norm{ \bar{q}_u } \leq c$, the density process of $Q^{\bar{q}^n}$ is bounded and the sequence of positive random variables $(D_{0,T}^{\bar{\beta}^n}dQ^{\bar{q}^n}/dP)$ converges $P$-almost surely to $D_{0,T}^{\bar{\beta}} dQ^{\bar{q}}/dP$.
	Furthermore, for any $p >1$ it holds
	\begin{equation}\label{bound_density} 
		 \begin{split}
			  E\left[ \abs{\frac{dQ^{\bar{q}^n}}{dP}}^p \right]&= E\left[\exp\left( \int_0^Tp\bar{q}^n_u dW_u - \frac{1}{2}\int_0^T \norm{p\bar{q}^n_u}^2du +\frac{p(p-1)}{2}\int_0^T \norm{\bar{q}^n_u}^2du\right)\right]\\
			  &\leq \exp\left( \frac{p(p-1)}{2}c^2T \right).
		 \end{split}
	\end{equation}
	Hence $(D_{0,T}^{\bar{\beta}^n}dQ^{\bar{q}^n}/dP)$ is uniformly integrable.
	Therefore, since $X$ is bounded it holds
	\begin{equation*}
		\lim_{n\rightarrow \infty}E_{Q^{\bar{q}^n}}\left[ D_{t,T}^{\bar{\beta}^n} X \Mid \mathcal{F}_t \right] = E_{Q^{\bar{q}}}\left[ D_{t,T}^{\bar{\beta}}X \Mid \mathcal{F}_t \right].
	\end{equation*}
	Let us show that
	\begin{equation} \label{eq:second_convergence}
		\lim_{n \rightarrow \infty} E_{Q^{\bar{q}^n}}\left[ \int_t^{T} D_{t,u}^{\bar{\beta}^n}f^\ast_u(\bar{\beta}^n_u,\bar{q}^n_u)\,du \mmid \mathcal{F}_t \right] = 
		E_{Q^{\bar{q}}}\left[ \int_t^{T}D_{t,u}^{\bar{\beta}} f^\ast_u(\bar{\beta}_u,\bar{q}_u)\,du \mmid \mathcal{F}_t \right].
	\end{equation}
	For almost all $\omega \in \Omega$ and $t\leq u\leq T$, by definition of $\bar{\beta}^n$ and $\bar{q}^n$, it holds $(\bar{\beta}^n_u(\omega),\bar{q}^n_u(\omega))=(\bar{\beta}_u(\omega), \bar{q}_u(\omega))$ for $n$ large enough.
	Hence, the sequence $(D_{t,u}^{\bar{\beta}^n}f^\ast_u(\bar{\beta}^n_u,\bar{q}^n_u))$ converges $P\otimes dt$-almost surely to $D_{t,u}^{\bar{\beta}} f^\ast_u(\bar{\beta}_u,\bar{q}_u)$.
	Since the processes $\bar{\beta}$ and $\bar{q}$ are bounded, by Equation \eqref{eq:f_attained} and the linear growth assumption on $f$, we can find two positive numbers $C_1$ and $C_2$ such that
	\begin{equation}
		\label{eq:estimate_f-ast}
		\int_0^T\abs{f_u^\ast(\bar{\beta}_u,\bar{q}_u)}\,du \leq C_1\int_0^T \abs{ Y_u} \,du + C_2\int_0^T\norm{ Z_u }\,du.  
	\end{equation}
	It is known that if $X$ is bounded and $f$ is Lipschitz, then the solution $(Y,Z)$ of the BSDE is such that $Y$ is bounded and $\int Z\,dW$ is in BMO, see for instance \citep{kobylanski01} and \citep[Proposition 7.3]{barrieu01}\footnote{Notice that in \citep{barrieu01} , the generator does not depend on $y$, but the same proof carries over to the general case as mentioned in \citep{ravanelli01}.}.
	Equation \eqref{eq:estimate_f-ast} and $BMO\subseteq \mathcal{H}^p$ for all $1\leq p< \infty$, see \citep{Kazamaki01}, together with  H\"older's inequality imply
	\begin{multline*}
		E\left[\left( \frac{dQ^{\bar{q}^n}}{dP}\int_0^T \abs{f^\ast_u(\bar{\beta}^n_u, \bar{q}^n_u)}\,du \right)^2 \right]\\
		\leq \tilde{C}_1E\left[ \left(\frac{dQ^{\bar{q}^n}}{dP} \right)^2 \right] + \tilde{C}_2 E\left[ \left( \frac{dQ^{\bar{q}^n}}{dP}\right)^4 \right]^{1/2}E\left[ \left(\int_0^T\norm{ Z_u }^2\,du \right)^2 \right]^{1/2} \leq C,
	\end{multline*}
	where $C$ is a positive real number independent of $n$.
	Recalling that $D^{\beta^n}$ is bounded, we get the required uniform integrability to derive \eqref{eq:second_convergence}.
	Now, from Equation \eqref{Y_q_bar} and since $f^\ast$ is positive, we obtain
	\begin{align*}
		Y_t  & = E_{Q^{\bar{q}}}\left[ D_{t,T}^{\bar{\beta}} X - \int_t^{T} D_{t,u}^{\bar{\beta}}f^\ast_u(\bar{\beta}_u,\bar{q}_u)\,du \mmid \mathcal{F}_t \right]\\
			 & = \lim_{n \rightarrow \infty}E_{Q^{\bar{q}^n}}\left[ D_{t,T}^{\bar{\beta}^n} X -  \int_t^{T} D_{t,u}^{\bar{\beta}^n}f^\ast_u(\bar{\beta}^n_u,\bar{q}^n_u)\,du \mmid \mathcal{F}_t \right] \\
			 & \leq \esssup_{(\beta,q) \in \mathcal{D}\times \mathcal{Q}}E_{Q^q}\left[ D_{t,T}^\beta X - \int_t^T D_{t,u}^\beta f^\ast_u(\beta_u,q_u)\,du \mmid \mathcal{F}_t \right].
	\end{align*}
	Together with Equation \eqref{first_Y_ineq}, this concludes the proof.
\end{proof}
\begin{remark}
	Equation \eqref{Y_q_bar} enables us already to obtain the representation of the $g$-expectation with respect to measure with square integrable densities. 
	This is a well-known result.
	The role of the subsequent localization procedure is to prove that the representation can, in fact, be written with respect to measures with bounded densities.
	This turns out to be important for the representation in the non-bounded case, since we work on the $\mathcal{H}^1$-$L^\infty$ duality.
\end{remark}
Considering a more general driver, we can build on the result above to represent the minimal supersolution functional defined on the set of essentially bounded random variables.
\begin{theorem}
	\label{thm:rep-bounded1}
	Let $g$ be a driver satisfying \ref{jconv}, \ref{lsc} and \ref{pos}.
	Then, the operator $\mathcal{E}_0: L^{\infty} \rightarrow \mathbb{R} \cup \set{\infty}$ admits the dual representation
	\begin{equation}
		\label{eq:rep-bounded-0}
		\mathcal{E}_0(X) = \sup_{(\beta,q) \in \mathcal{D}\times \mathcal{Q}}\Set{E_{Q^q}\left[ D_{0,T}^\beta X \right]- \alpha_{0,T}(\beta,q)},
	\end{equation}
	where the penalty function $\alpha$ is given by
	\begin{equation}
		\alpha_{t,s}(\beta,q):=E_{Q^q}\left[ \int_{t}^{s} D_{t,u}^\beta g^\ast_u(\beta_u,q_u)\,du \mmid \mathcal{F}_t \right],\quad (\beta,q) \in \mathcal{D}\times \mathcal{Q}
		\label{eq:rep-bounded-penalty}
	\end{equation}
	for every $0\leq t\leq s\leq T$.
	
	In addition, for any $t \in [0,T]$, and $X \in L^\infty$ such that $\mathcal{E}_0(X) < \infty $, 
	\begin{equation}
		\label{eq:rep-bounded-1}
		\mathcal{E}_t(X) = \esssup_{(\beta,q) \in \mathcal{D}\times \mathcal{Q}}\Set{E_{Q^q}\left[ D_{t,T}^\beta X \Mid \mathcal{F}_t \right]- \alpha_{t,T}(\beta,q)}.
	\end{equation}
\end{theorem}
\begin{proof}
	\begin{enumerate}[fullwidth]
		\item[\emph{First inequality:}]
			Let $X$ be a bounded terminal condition. 
			If $\mathcal{A}(X) \neq \emptyset $, then we fix a supersolution $(Y,Z)\in \mathcal{A}(X)$.
			Let $t \in [0,T]$ and $(\beta, q) \in \mathcal{D}\times\mathcal{Q}$.
			Let us define the localizing sequence of stopping times $(\tau_n)$ by
			\begin{equation*}
				\tau_n: = \inf\Set{ s >t: \abs{ \int_t^sZ_u\,dW_u } > n }\wedge T, \quad n\in \mathbb{N}.
			\end{equation*}
			We apply It\^o's formula to $\bar{Y}_u = D_{t,u}^\beta Y_u$ for $u\geq t$.
			Since $(Y, Z)$ satisfies the equivalent formulation \eqref{eq:equiv_supersol}, there exists a nondecreasing process $K$ such that
			\begin{equation*}
				d\bar{Y}_u = - \beta_u D_{t,u}^\beta Y_u\,du + D_{t,u}^\beta \left(Z_u\,dW_u - g_u(Y_u,Z_u)\,du - dK_u \right). 
			\end{equation*}
			Hence, $K$ being nondecreasing, it follows
			\begin{equation*}
				\bar{Y}_{\tau_n} - \bar{Y}_{t} \leq \int_{t}^{\tau_n}D_{t,u}^\beta \left( -\beta_uY_u + q_uZ_u - g(Y_u,Z_u) \right)\,du + \int_{t}^{\tau_n}D_{t,u}^\beta Z_u\,dW^{Q^q}_u.
			\end{equation*}
			Applying Girsanov's theorem, it follows that $\int_t^{\cdot\wedge \tau_n} D_{t,u}^\beta Z_u\,dW^{Q^q}_u$ is a $Q^q$-martingale between $t$ and $T$.
			Taking conditional expectation on both sides, using the definition of $g^\ast$, the facts that $Y_{\tau_n} \geq E\left[ X \mid \mathcal{F}_{\tau_n}\right]$ and $g \geq 0$, we are led to 
			\begin{equation*}
				Y_{t} \geq E_{Q^q}\left[ D_{t,\tau_n}^\beta E\left[ X \mid \mathcal{F}_{\tau_n} \right]  - \int_{t}^{\tau_n}D_{t,u}^\beta g_u^\ast(\beta_u,q_u)\,du \mmid \mathcal{F}_t \right].
			\end{equation*}
			Since $X$ is bounded, taking the limit on the right hand side we obtain by dominated convergence
			\begin{equation*}
				Y_{t} \geq E_{Q^q}\left[ D_{t,T}^\beta X   - \int_{t}^{T}D_{t,u}^\beta g_u^\ast(\beta_u,q_u)\,du \mmid \mathcal{F}_t \right],
			\end{equation*}			
			so that taking the supremum with respect to $\beta$ and $q$ and by the fact that $Y$ was chosen arbitrary, we have
			\begin{equation}
			\label{eq:bounded-first}
				\mathcal{E}_t(X) \geq \esssup_{(\beta,q) \in \mathcal{D}\times \mathcal{Q}}E_{Q^q}\left[ D_{t,T}^\beta X  - \int_t^{T}D_{t,u}^\beta g_u^\ast(\beta_u,q_u)\,du \mmid \mathcal{F}_t\right].
			\end{equation}
			If $\mathcal{A}(X)  = \emptyset$, then Equation \eqref{eq:bounded-first} is obvious. 
		\item[\emph{Second inequality:}]
			Let $n\in \mathbb{N}$, and define
			\begin{equation*}
				g^n(y,z) := \sup_{\set{\abs{ \beta}\leq n; \norm{ q } \leq n}}\Set{ -\beta y + q z - g^*(\beta,q) }.
			\end{equation*}
			For every $n\in \mathbb{N}$, the function $g^n$ satisfies the assumptions of Proposition \ref{g_expectation}.
			Namely, $g^n$ is proper, has linear growth in $y$ and $z$ and satisfies \ref{jconv}, \ref{lsc} and \ref{pos}.
			Moreover, the sequence $(g^n)$ is nondereasing and by the Fenchel-Moreau theorem, it converges pointwise to $g$.
			By Proposition \ref{g_expectation}, the solution $(Y^n,Z^n)$ of the BSDE with generator $g^n$ and terminal condition $X$ has the dual representation
			\begin{equation*}
				Y^n_t = \esssup_{(\beta,q) \in \mathcal{D}\times \mathcal{Q}}E_{Q^q}\left[ D_{t,T}^\beta X - \int_t^T D_{t,u}^\beta g^{n,\ast}_u(\beta_u,q_u)\,du \mmid \mathcal{F}_t \right].
			\end{equation*}
			Let us denote by $(\bar{Y}^n, \bar{Z}^n)$ the minimal supersolution\footnote{As explained in Remark \ref{rem:sol_minsupersol} we cannot ensure at this point that $Y^n = \bar{Y}^n$.} of the BSDE with driver $g^n$ and terminal condition $X$.
			Since for every $n \in \mathbb{N}$ we have $g^n\leq g$, it holds $g^{n,\ast} \geq g^\ast$, and, by minimality of $\bar{Y}^n$ we have $\bar{Y}^n_t \leq Y^n_t$.
			Thus, for all $n \in \mathbb{N}$
			\begin{equation}
			\label{eq:bounded-second-appr}
				\bar{Y}^n_t \leq \esssup_{(\beta,q) \in \mathcal{D}\times \mathcal{Q}}E_{Q^q}\left[ D_{t,T}^\beta X - \int_t^T D_{t,u}^\beta g^{\ast}_u(\beta_u,q_u)\,du \mmid \mathcal{F}_t \right].
			\end{equation}
			If $t = 0$, taking the limit as $n$ goes to infinity and using the monotone stability of minimal supersolutions of BSDEs, see Theorem \ref{thm:stability-withoutDEC}, we obtain
			\begin{equation*}
				\mathcal{E}_0(X) \leq \sup_{(\beta,q) \in \mathcal{D}\times \mathcal{Q}}E_{Q^q}\left[ D_{0,T}^\beta X - \int_0^T D_{0,u}^\beta g^{\ast}_u(\beta_u,q_u)\,du \right].
			\end{equation*}
			Therefore Equation \eqref{eq:rep-bounded-0} holds true.
			If $t \in [0,T]$ and $\mathcal{E}_0(X) < \infty $, then it holds, by monotonicity, $\lim_{n}\bar{Y}^n_0 < \infty$. 
			Hence, taking the limit in Equation \eqref{eq:bounded-second-appr}, by Theorem \ref{thm:stability-withoutDEC} we have
			\begin{equation*}
				\mathcal{E}_t(X) \leq \esssup_{(\beta,q) \in \mathcal{D}\times \mathcal{Q}}E_{Q^q}\left[ D_{t,T}^\beta X  - \int_t^{T}D_{t,u}^\beta g_u^\ast(\beta_u,q_u)\,du \mmid \mathcal{F}_t\right],
			\end{equation*}
			which ends the proof.
	\end{enumerate}
\end{proof}	
In the next corollary, we extend the result of Theorem \ref{thm:rep-bounded1} by giving conditions under which the representation is valid on the whole space $L^\infty$ even in the dynamic case. 
\begin{corollary}
\label{thm:rep-bounded}
	Let $g$ be a driver satisfying \ref{jconv}, \ref{dec}, \ref{lsc} and \ref{pos}.
	Then either $\mathcal{E}_t(X) \equiv + \infty$ for all $X  \in L^\infty$, $t \in [0,T]$, or $\mathcal{E}: L^{\infty} \rightarrow \mathcal{S}$ admits the dual representation
	\begin{equation}
		\label{eq:rep-bounded}
		\mathcal{E}_t(X) = \sup_{(\beta,q) \in \mathcal{D}_+\times \mathcal{Q}}\Set{E_{Q^q}\left[ D_{t,T}^\beta X \Mid \mathcal{F}_t \right]- \alpha_{t,T}(\beta,q)},\quad X \in L^{\infty},\, t \in [0,T],
	\end{equation}
	where the penalty function $\alpha$ is defined in Theorem \ref{thm:rep-bounded1}.
\end{corollary}
\begin{proof}
	If for every $X \in L^\infty$ the set $\mathcal{A}(X)$ is empty, then the domain of $\mathcal{E}$ is empty. 
	On the other hand, if there exists $\xi \in L^\infty$ such that $\mathcal{A}(\xi) \neq \emptyset$, then $\mathcal{A}(X) \neq \emptyset $ for all $X \in L^\infty$. 
	In fact, using $-\norm{\xi}_\infty \le \xi $ we have $\mathcal{A}(-\norm{\xi}_\infty) \neq \emptyset$ and by \ref{dec}, see the arguments of the proof of Proposition \ref{thm:properties}, we have $\mathcal{A}(-\norm{\xi}_\infty + c) \neq \emptyset $ for all $c\ge 0$. 
	Hence $\mathcal{A}(X) \neq \emptyset$ for all $X \in L^\infty $, since $X \le \norm{X}_\infty$ and $\mathcal{A}(\norm{X}_\infty) \neq \emptyset $ for all $X \in L^\infty $. 

	The rest of the proof is similar to that of Theorem \ref{thm:rep-bounded1}.
	Because $g$ satisfies \ref{dec}, the domain of $g^\ast$ is concentrated on $ \mathbb{R}_+ \times \mathbb{R}^d $, so that the representation can be restricted to $  \mathcal{D}_+ \times \mathcal{Q} $.
\end{proof}
\begin{remark}
\label{rem:sol_minsupersol}
	For a given BSDE, it is not a priori clear that the minimal supersolution solution and the solution agree, since the measure induced by the process $K$ appearing in the definition of the minimal supersolution can be singular to the Lebesgue measure.
	Proposition \ref{g_expectation} and Corollary \ref{thm:rep-bounded} show that if the terminal condition is bounded and the generator is of linear growth both in $y$ and $z$, then the minimal supersolution of a BSDE coincides with its solution. 
	In particular, $\mathcal{E}(X)$ is a continuous process, compare \citep[Proposition 4.4]{DHK1101}.
\end{remark}

%% file: chap04.tex
\subsection{The Extension to $\mathcal{H}^1$}
The goal of this section is to extend the dual representation of $\mathcal{E}_0$ to the space $\mathcal{H}^1$.
We define
\begin{equation*}
	\mathcal{SQ} = \Set{ M \in L^\infty_{++}: E[ M ] \leq 1 }.
\end{equation*}
We denote by $\mathcal{E}^\ast_0 $ the convex conjugate of $\mathcal{E}_0$, defined as
\begin{equation*}
	\mathcal{E}^\ast_{0}(M) := \sup_{X \in \mathcal{H}^1}\Set{ E\left[ M X  \right] - \mathcal{E}_0(X) }, \quad M \in L^\infty.
		\label{eq:rest_on_SQ_penalty}
\end{equation*}
The following lemma is a consequence of the Fenchel-Moreau theorem and the structural properties of $\mathcal{E}_0$.
\begin{lemma}\label{lem:rep-sq}
	Let $g$ be a driver satisfying \ref{jconv}, \ref{dec}, \ref{lsc}, \ref{pos} and such that $\mathcal{E}_0$ is proper.\footnote{$\mathcal{E}_0$ is proper for instance if there exists $y_0 \in \mathbb{R}$ with $g(y_0,0) = 0$. In fact, in that case, the pair $(y_0,0)$ is in $\mathcal{A}(y_0)$ and therefore $\mathcal{E}_0(y_0) \le y_0 < \infty$. And by \ref{pos},  $\mathcal{E}_0(X) \ge E[X] > - \infty$ for all $X \in \mathcal{H}^1 $. } 
	Then, the operator $\mathcal{E}_0: \mathcal{H}^1 \rightarrow ]-\infty ,\infty]$ is $\sigma(\mathcal{H}^1, L^\infty)$-lower semicontinuous, and admits the dual representation
	\begin{equation}
		\label{eq:rest_on_SQ}
		\mathcal{E}_0(X) = \sup_{M \in \mathcal{SQ}}\Set{E[MX] - \mathcal{E}_0^\ast(M)}, \quad X \in \mathcal{H}^1.
	\end{equation}
\end{lemma}
\begin{proof}
	$\mathcal{E}_0$ is proper, convex since $g$ fulfills \ref{jconv}, and $\sigma (L^1, L^\infty)$-lower semicontinuous by \citep[Theorem 4.9]{DHK1101} and therefore, since $\mathcal{H}^1 \subseteq L^1$, it is $\sigma(\mathcal{H}^1,L^\infty)$-lower semicontinuous.
	By the Fenchel-Moreau theorem, it follows 
		\begin{equation}
			\mathcal{E}_0(X) = \sup_{M \in L^\infty} \Set{E\left[  M X \right]-\mathcal{E}^\ast_{0}(M)}, \quad X \in \mathcal{H}^1.
		\label{E:Fenchel}
	\end{equation}
	A standard argument shows that we can restrict the previous supremum from $L^\infty$ to $\mathcal{SQ}$.
	On the one hand, let $M \in L^\infty$ with $E[M] > 1 $ and $\xi_0 \in \mathcal{H}^1 $ such that $\mathcal{E}_0(\xi_0) < \infty $.
	By cash-subadditivity, see Proposition \ref{thm:properties}, it holds
	\begin{align*}
		\mathcal{E}_0^\ast(M)  & \geq \sup_{n \in \mathbb{N}}\Set{E[M(\xi_0 + n)] - \mathcal{E}_0(\xi_0 + n)}\\
							   & \geq \sup_{n \in \mathbb{N}}\Set{n\left( E[M] - 1 \right) } + E[M\xi_0] - \mathcal{E}_0(\xi_0) = +\infty.
	\end{align*}
	On the other hand, let $M \in L^\infty \setminus L^\infty_+$ and $\xi_0 \in \mathcal{H}^1 $ such that $\mathcal{E}_0(\xi_0) < \infty $.
	There is $\bar{X} \in \mathcal{H}^1_+$ such that $E\left[ M \bar{X} \right] < 0$ since $L^\infty_+$ is the polar cone of $\mathcal{H}^1_+$.
	By monotonicity of $\mathcal{E}_0$, we have $\mathcal{E}_0(-n\bar{X} + \xi_0) \leq \mathcal{E}_0(\xi_0)$.
	Hence, 
	\begin{align*}
		\mathcal{E}^\ast_{0}(M) &\geq \sup_{n\in \mathbb{N}}\Set{nE\left[- M \bar{X} \right] + E[M\xi_0] - \mathcal{E}_0(-n\bar{X} + \xi_0)}\\
		                        & \geq \sup_{n \in \mathbb{N}}\Set{ nE\left[- M \bar{X}  \right] } + E[M\xi_0] - \mathcal{E}_0(\xi_0) = +\infty.		
	\end{align*}
	Therefore, we have
	\begin{equation*}
		\mathcal{E}_0(X) = \sup_{M \in L^\infty_+:\,\, E[M]\leq 1}\Set{E\left[ MX \right] - \mathcal{E}^\ast_0(M)}.
	\end{equation*}
	Now, let $M \in L^\infty_+ $ such that $E[M] \le 1$, and for all $\lambda \in (0,1)$, we put $M^\lambda = (1 - \lambda)M + \lambda$.
	Then, $M^\lambda \in \mathcal{SQ}$.
	Since for any $X \in \mathcal{H}^1$ we have $\mathcal{E}_0(X) \ge E[X] $, it follows from the definition of $\mathcal{E}_0^\ast $ that $\mathcal{E}_0^\ast(1) \leq 0 $ so that by convexity, it holds $\limsup_{\lambda \rightarrow 0}\mathcal{E}^\ast_0(M^\lambda) \leq \mathcal{E}^\ast_0(M)$.
	Let $X \in \mathcal{H}^1$, applying dominated convergence theorem to $(M^\lambda X)$ implies
	\begin{equation*}
		E[MX] - \mathcal{E}^\ast_0(M) \leq \liminf_{\lambda \rightarrow 0} \Set{ E\left[M^\lambda X \right] - \mathcal{E}^\ast_0(M^\lambda)}.
	\end{equation*} 
	Hence, $\mathcal{E}_0(X) \leq \sup_{M \in \mathcal{SQ}} \set{E[MX] - \mathcal{E}^\ast_0(M)} $. 
	The other inequality follows by sets inclusion.
	Thus, Equation \eqref{eq:rest_on_SQ} holds true.
\end{proof}
We observe that there is a relationship between the sets $\mathcal{SQ}$ and $\mathcal{D}_+\times \mathcal{Q} $, and the dual representation of $\mathcal{E}_0$.
\begin{remark}
	Any element of $\mathcal{SQ} $ may be parametrized by elements of $\mathcal{D}_+\times \mathcal{Q} $ and vice versa.
	Indeed, for every $M \in \mathcal{SQ}$, since $M/E[M]$ is a strictly positive random variable with expectation 1, there exists a unique process $q \in \mathcal{Q}$ such that $M^q_T = M/E[M]$, with $M^q_t = \exp(\int_0^t q_u\,dW_u - \frac{1}{2}\int_0^t \norm{q}^2_u\,du) $ and taking $\beta \in \mathcal{D}_+$ such that $\exp(-\int_0^T \beta_s ds) = E[M] \in (0,1]$, we have $ M = \exp(-\int_0^T \beta_s ds)M^q_T$.
	Conversely, given $(\beta,q)\in \mathcal{D}_+ \times \mathcal{Q}$, it holds $\exp( -\int_{0}^{T}\beta_s ds )\exp( \int_0^Tq_u\,dW_u - \frac{1}{2}\int_0^T\norm{q}^2_u\,du ) \in \mathcal{SQ}$.
	This underlines the importance of working with probability measures with bounded densities in the previous section.
\end{remark}

\begin{remark}
\label{rem:intro-alphamin}
	To every $M \in \mathcal{SQ}$ corresponds a unique $q \in \mathcal{Q}$. Hence, for all $X \in L^\infty$, Corollary \ref{thm:rep-bounded} yields
	\begin{align*}
		\mathcal{E}_0(X) & = \sup_{(\beta, q) \in \mathcal{D}_+ \times \mathcal{Q}}\Set{ E\left[\frac{dQ^q}{dP}D^{\beta}_{0,T}X \right] - E_{Q^q}\left[ \int_0^TD^{\beta}_{0,u}g_u^\ast(\beta_u, q_u)\,du \right] }\\
                 & = \sup_{M \in \mathcal{SQ}}\sup_{\set{\beta\in \mathcal{D}_+:\,\, D^{\beta}_{0,T} = E[M]}}\Set{ E\left[MX \right] - E_{Q^q}\left[ \int_0^TD^{\beta}_{0,u}g_u^\ast(\beta_u, q_u)\,du \right] }\\
                 & = \sup_{M \in \mathcal{SQ}}\Set{ E\left[MX \right] - \alpha_{min}(M) },
	\end{align*}
	for the penalty function 
	\begin{equation}
	\label{eq:min-penalty}
		\alpha_{min}(M) := \inf_{\set{\beta\in \mathcal{D}_+:\,\, D^{\beta}_{0,T} = E[M]} }E_{Q^q}\left[ \int_0^TD^{\beta}_{0,u}g_u^\ast(\beta_u, q_u)\,du \right]
	\end{equation}
	defined on $\mathcal{SQ}$.
\end{remark}
We may now present the main result of this section, the extension to $\mathcal{H}^1$ of the dual representation Theorem \ref{thm:rep-bounded}.
\begin{theorem}\label{thm:rep-nonbounded}
	Let $g$ be a driver satisfying \ref{jconv}, \ref{dec}, \ref{lsc} and \ref{pos} and such that $\mathcal{E}_0$ is proper.
	Then the operator $\mathcal{E}_0: \mathcal{H}^1 \rightarrow ] -\infty, + \infty]$ admits the dual representation
	\begin{equation}\label{eq:rep-nonbounded}
		\mathcal{E}_0(X) = \sup_{(\beta,q) \in \mathcal{D}_+ \times \mathcal{Q}}\Set{E_{Q^q}\left[ D_{0,T}^\beta X \right]- \alpha_0(\beta,q)},\quad X \in \mathcal{H}^{1},
	\end{equation}
	where
	\begin{equation}
		\alpha_0(\beta,q):=E_{Q^q}\left[ \int_{0}^{T} D_{0,u}^\beta g^\ast_u(\beta_u,q_u)\,du \right],\quad (\beta,q) \in \mathcal{D}_+ \times \mathcal{Q}.
		\label{eq:rep-nonbounded-penalty}
	\end{equation}
\end{theorem}
\begin{proof}
	Due to Lemma \ref{lem:rep-sq} and Remark \ref{rem:intro-alphamin}, it suffices to show that $\mathcal{E}^\ast_0 = \alpha_{min}$ on $\mathcal{SQ}$, where $\alpha_{min}$ is the penalty function defined by Equation \eqref{eq:min-penalty}.
	\begin{enumerate}[label= ,fullwidth]
		\item\emph{ First inequality.} \label{step:first_inequalityH1}
		For all $X \in \mathcal{H}^1$, it holds
		\begin{equation}\label{eq:rep_h1_first-inequality}
			\mathcal{E}_0(X) \geq \sup_{(\beta,q) \in \mathcal{D}_+ \times \mathcal{Q}}E_{Q^q}\left[ D^{\beta}_{0,T}X - \int_0^TD^{\beta}_{0,u}g^{\ast}_u(\beta_u,q_u)\,du \right].
		\end{equation}
		In fact, let $X \in \mathcal{H}^1$. If $\mathcal{A}(X) = \emptyset$, then the result is trivial. 
		Suppose that $\mathcal{A}(X) \neq \emptyset$, and take $(Y,Z) \in \mathcal{A}(X)$. Let $(\beta,q) \in \mathcal{D}_+ \times\mathcal{Q}$, 
		arguing exactly like in the first part of the proof of Theorem \ref{thm:rep-bounded1} we obtain a localizing sequence of stopping times $(\tau_n)$ such that
		\begin{equation}\label{eq:conditional_expectation}
			Y_0 \geq E_{Q^q}\left[ D^{\beta}_{0,\tau_n} E[X \mid \mathcal{F}_{\tau_n}] - \int_0^{\tau_n }D^{\beta}_{0,u}g_u^\ast(\beta_u,q_u)\,du \right] \quad \text{for all } n \in \mathbb{N}.
		\end{equation}
		Since $X \in \mathcal{H}^1 $, the sequence of martingales $(N^n)$ given by $N^n_t := E[E[X \mid \mathcal{F}_{\tau_n}] \mid \mathcal{F}_{t}] = E[X \mid \mathcal{F}_{\tau_n\wedge t} ]$ is in $\mathcal{H}^1$, and is such that $\bigl( \sup_{t \in [0,T]} \abs{ N^n_t } \bigr)_n$ is uniformly integrable. 
		Therefore, by \citep[Theorem 4.9]{delbaen03}, see also \citep[Lemma 2.5]{Kazamaki01}, $(N^n)$ admits a subsequence again denoted by $(N^n)$ which converges weakly in $\mathcal{H}^1$. 
		Thus, the sequence of products $(D^\beta_{0,\tau^n}N^n_T) $ converges weakly in $\mathcal{H}^1$ to $D^\beta_{0,T}X$, since $(D^\beta_{0,\tau^n})$ is bounded by $1$.
		Now, as a consequence of the boundedness of the martingale $M^q_t = E[dQ^q/dP\mid \mathcal{F}_t ] $, the function $X \mapsto E[M^q_TX ]$ from $\mathcal{H}^1$ to $\mathbb{R}$ is linear and continuous, and therefore $\sigma(\mathcal{H}^1, BMO) $-continuous.
		Hence, taking the limit on both sides of Equation \eqref{eq:conditional_expectation} leads to
		\begin{equation*}
			Y_0 \geq E_{Q^{q}}\left[ D^{\beta}_{0,T} X  - \int_0^{T} D^{\beta}_{0,u}g_u^\ast(\beta_u,q_u)\,du \right].
		\end{equation*}
		This implies, by means of Remark \ref{rem:intro-alphamin}, that
		\begin{equation*}
			\mathcal{E}_0(X) \ge \sup_{M \in \mathcal{SQ} }\Set{ E\left[MX \right] - \alpha_{min}(M) },
		\end{equation*}
		that is, for every $M \in \mathcal{SQ}$ we have $\alpha_{min}(M) \ge E\left[MX \right] - \mathcal{E}_0(X) $ so that taking the supremum with respect to $X \in \mathcal{H}^1$, we obtain by definition of $\mathcal{E}^\ast$
		\begin{equation*}
			\alpha_{min}(M) \geq \mathcal{E}^\ast_{0}(M).
		\end{equation*}
		\item \emph{Second inequality.} 
		The main argument for the second inequality is to show that the penalty function $\alpha_{min}$ defined by Equation \eqref{eq:min-penalty}  is minimal, that is, 
		\begin{equation*}
				\mathcal{E}^\ast_{L^\infty}(M) := \sup_{X \in L^\infty}\Set{E\left[ M X \right]-\mathcal{E}_0(X)} = \alpha_{min}(M), \quad M \in \mathcal{SQ}.
		\end{equation*}
		In fact, that would imply
		\begin{equation*}
			\mathcal{E}^\ast_{0}(M) \ge \mathcal{E}^\ast_{L^{\infty}}(M) = \alpha_{min}(M), 
		\end{equation*}
		where the first inequality is obtained by sets inclusion.
		To that end, it suffices to show that for every $c \geq 0$ the set $\{M \in \mathcal{SQ}: \alpha_{min}(M) \leq c \}$ is convex and closed in $L^1$, since by convexity, it would then be $\sigma(L^1, L^\infty)$-closed and therefore $\sigma(L^\infty, L^\infty)$-closed. 
		\begin{enumerate}[label= ,fullwidth]
			\item \emph{Convexity:}  
			Let $\lambda \in [0,1]$, $M^1, M^2  \in \mathcal{SQ}$ and $q^i \in \mathcal{Q}$ such that $M^{q^i}_T = M^i/E[M^i]$, $i = 1,2$. Put $M^{\lambda} = \lambda M^1 + (1 -\lambda)M^2$. For a given $\varepsilon > 0$, there exists $\beta^i \in \mathcal{D}_+	$ such that $D^{\beta^i}_{0,T} = E[M^i]$ and
			\begin{equation*}
				\varepsilon + \alpha_{min}(M^i) \geq E_{Q^{q^i} } \left[ \int_0^TD^{\beta^i}_{0,u}g^\ast(\beta^i_u,q^i_u)\,du \right].
			\end{equation*}
			Applying It\^o's formula to $\log\bigl(\lambda M^{q^1}_tD^{\beta^1}_t + (1 - \lambda)M^{q^2}_tD^{\beta^2}_t \bigr)$ such as in the proof of \citep[Lemma 2.1]{BDH10} we have
			\begin{equation*}
				\lambda M^{q^1}_tD^{\beta^1}_{0,t} + (1 - \lambda)M^{q^2}_tD^{\beta^2}_{0,t} = \exp\left(\int_0^tq^{\lambda}_u\,dW_u - \frac{1}{2}\int_0^t \norm{ q_u^{\lambda} }^2\,du - \int_0^t\beta_u^{\lambda}\,du \right)	= M^{q^{\lambda}}_tD^{\beta^{\lambda}}_{0,t}	
			\end{equation*}
			 and $D^{\beta^{\lambda}}_{0,T} = E[M^{\lambda}]$, with
			\begin{equation*}
				q^{\lambda}_t = \frac{\lambda M^{q^1}_tD^{\beta^1}_{0,t}q^1_t + (1-\lambda)M^{q^2}_tD^{\beta^2}_{0,t}q^2_t}{\lambda M^{q^1}_tD^{\beta^1}_{0,t} + (1 - \lambda)M^{q^2}_tD^{\beta^2}_{0,t}},\quad
				\beta^{\lambda}_t = \frac{\lambda M^{q^1}_tD^{\beta^1}_{0,t}\beta^1_t + (1-\lambda)M^{q^2}_tD^{\beta^2}_{0,t}\beta^2_t}{\lambda M^{q^1}_tD^{\beta^1}_{0,t} + (1 - \lambda)M^{q^2}_tD^{\beta^2}_{0,t}}.
			\end{equation*}
 			This follows from the facts that $M^{q^{\lambda}}_TE[M^{\lambda}] = M^{\lambda} = M^{q^{\lambda}}_TD^{\beta^{\lambda}}_{0,T}$ and $M^{q^\lambda} > 0$.     
			Therefore, joint convexity of $g^\ast$ and the definition of $(\beta^\lambda,q^\lambda)$ lead us to
			\begin{align*}
				2\varepsilon  + \lambda\alpha_{min}(M^1)  + (1-\lambda) \alpha_{min}(M^2)  
				 &\geq E\left[ \int_0^T\left(\lambda M^{q^1}_u D^{\beta^1}_{0,u} + (1-\lambda)M^{q^2}_uD^{\beta^2}_{0,u}  \right)g_u^\ast(\beta^{\lambda}_u, q^{\lambda}_u)\,du \right]\\
				  &= E_{Q^{q^{\lambda}}}\left[ \int_0^TD^{\beta^{\lambda}}_{0,u}g_u^\ast(\beta^{\lambda}_u,q^{\lambda}_u)\,\,du \right].
			\end{align*}
			Therefore, taking first the infimum for $\beta \in \mathcal{D}_+$ such that $D^{\beta}_{0,T} = E[M^{\lambda}]$ on the right hand side, and then the limit on the left hand side as $\varepsilon$ goes to $0$  we have
			\begin{equation*}
				\lambda\alpha_{min}(M^1) + (1-\lambda) \alpha_{min}(M^2) \geq \alpha_{min}(M^{\lambda}).
			\end{equation*}
			\item \emph{Closedness:}  			
			Let $c \ge 0$ and $(M^n)$ be a sequence in $\mathcal{SQ}$ converging to $M \in \mathcal{SQ}$ in $L^1$ and such that $\alpha_{min}(M^n) \leq c$ for every $n\in \mathbb{N}$.
			Let us show that $\alpha_{min}(M) \le c$.
			For all $n \in \mathbb{N}$ let $q^n$ be such that $M^{q^n}_T = M^n/E[M^n]$ and $q$ be such that $M^q_T = M/E[ M]$. 
			Let $\varepsilon >0$ be fixed. 
			For every $n\in \mathbb{N}$, there exists $\beta^n \in \mathcal{D}_+$ such that $D^{\beta^n}_{0,T} = E[M^n]$ and
			\begin{equation*}
				\varepsilon + \alpha_{min}(M^n) \geq E_{Q^{q^n}}\left[ \int_0^TD^{\beta^n}_{0,u}g_u^\ast(\beta^n_u,q^n_u)\,du \right].
			\end{equation*}
			Since $(M^n)$ converges to $M$ in $L^1$, the sequence $(E[M^n])$ converges to $E[M]$, with $E[M^n] > 0$ and $E[M] >0$. 
			Therefore, $(M^{q^n}_t)$ converges to $M^{q}_t$ in $L^1$ for all $t \in [0,T]$.
			We also introduce the martingales $M^n_t := E\left[M^n\mid \mathcal{F}_t \right]$ and $M_t := E\left[M\mid \mathcal{F}_t \right] $, $t\in [0,T]$.
			We choose a fast subsequence $(M^{n,m})$ such that $	 	P(\abs{M^{n,m}_T - M_T} \ge 1) < 2^{-n}/{m}$
			and for all $m \in \mathbb{N}$, define the stopping time
			\begin{equation*}
				\tau^m := \inf\Set{ t \in [0,T]: \abs{M^{n,m}_t - M_t} \ge m \text{ for some } n }.
			\end{equation*}
			Then, $(\tau^m)$ is a localizing sequence of stopping times since
			\begin{align*}
				P(\tau^m = T) & \ge 1 - P(\abs{M^{n,m}_T - M_T} \ge 1 \text{ for some } n) - P(\abs{M_T} \ge m-1)\\
				              & \ge 1 - \frac{1}{m} - \frac{E[\abs{M_T}]}{m - 1} \longrightarrow 1.
			\end{align*}
			For every $m$, the sequence $(M^{n,m}_{\tau^m} - M_{\tau^m})$ is bounded, therefore $(M^{n,m}_{\tau^m})$ converges to $(M_{\tau^m})$ in $L^2$.
			It follows by Burkholder-Davis-Gundy and Doob's inequalities that there exists a positive constant $C$ such that
			\begin{align*}
				E\left[\int_0^{\tau^m} \abs{ M^{n,m}_0q^{n,m}_uM^{q^{n,m}}_u - M_0q_uM^q_u }^2\,du \right] & =	E\left[\langle M^{n,m} - M \rangle_{\tau^m}^{2} \right]\\ 
				&\leq CE\left[ \left(\sup_{t \in [0,\tau^m]}\abs{M^{n,m}_t - M_t}\right)^2 \right] \xrightarrow[n \to \infty]{} 0.
			\end{align*}
			Thus, up to a subsequence, $(q^{n,m}M^{q^{n,m}}1_{[0,\tau^m]})$ converges $P \otimes dt$-a.s. to $qM^q1_{[0,\tau^m]}$.
			But since the sequence of strictly positive martingales $\bigl( M^{q^{n,m}} \bigr)$ converges $P \otimes dt$-a.s. to $M^q >0$, it follows that 
			\begin{equation*}
				\lim_{n \rightarrow \infty} q^{n,m}1_{[0,\tau^m]} = q1_{[0,\tau^m]} \quad  P \otimes dt\text{-a.s.}
			\end{equation*}
			Since $(\tau^m)$ converges $P$-a.s. to $T$ we obtain, by a diagonalization argument, another subsequence again denoted $(q^n)$ which converges $P\otimes dt$-a.s. to $q$.
			As for the convergence of the  sequence $(\beta^{n})$, since $ (\exp(-\int_0^T\beta^n_u\,du )) = \left( E\left[ M^{n}\right] \right)$ converges to $E\left[M \right]$, it follows that the sequence $( \int_0^T\beta^n_u\,du )$ converges to $-\log(E[M])$, and $( E[ \int_0^T\beta^{n}_u\,du ])$ is uniformly bounded.
			Hence, we can apply a compactness argument, see for instance \citep[Theorem 1.4]{delbaen03} applied on the product space, to obtain a sequence $(\tilde{\beta}^n)$ in the asymptotic convex hull of $(\beta^n)$ which converges $P\otimes dt$ to a positive predictable process $\beta$. 
			In addition, $D^\beta_{0,T} = E[M]$ since the sequences $(\int_0^T\beta^n_u\,du )$ and $( \int_0^T\tilde{\beta}^n_u\,du ) $ converge to the same limit.
			Now applying Fatou's lemma, convexity and lower-semicontinuity of $g^\ast$ lead us to
			\begin{multline*}
				\varepsilon +	\liminf_{n \rightarrow \infty}\alpha_{min}(M^{n}) \geq  E\left[ \int_0^T \liminf_{n \rightarrow \infty}M^{q^{n}}_uD^{\beta^{n}}_{0,u}g^\ast_u(\beta^{n}_u , q^{n}_u )\,du \right] \\
				\geq E\left[ \int_0^T \liminf_{n \rightarrow \infty}M^{q^{n}}_uD^{\tilde{\beta}^{n}}_{0,u}g^\ast_u(\tilde{\beta}^{n}_u , q^{n}_u )\,du \right] = E\left[ \int_0^T M^{q}_uD^{\beta}_{0,u}g^\ast_u(\beta_u , q_u )\,du \right] \geq \alpha_{\min}(M).
			\end{multline*}
			Once again the result is obtained by letting $\varepsilon$ tend to $0$.
			\end{enumerate}	
	\end{enumerate}
\end{proof}
We recover the robust representation of coherent (cash-subadditive) risk measures.
\begin{corollary}
	\label{cor:rep-coherent}
	Under the assumptions of Theorem \ref{thm:rep-nonbounded}, if the generator $g$ is positive homogeneous in the sense that
	\begin{equation*}
		g(\lambda y,\lambda z) = \lambda g(y,z)\quad \text{for all $\lambda > 0$ and $(y,z) \in \mathbb{R} \times\mathbb{R}^d$},
	\end{equation*}
	then $\mathcal{E}_0$ is also positive homogeneous and the dual representation of $\mathcal{E}_0$ reduces to
	\begin{equation*}
		\label{E:coherent}
		\mathcal{E}_0(X) = \sup_{(\beta,q) \in \mathcal{D}_+\times \mathcal{Q}}E_{Q^q}\bigl[ D^{\beta}_{0,T}X \bigr], \quad X \in \mathcal{H}^1.
	\end{equation*}
\end{corollary}
\begin{proof}
		Let $\lambda$ be strictly positive, and $(\mathcal{E}_0(\lambda X), Z)$ the minimal supersolution in $\mathcal{A}(\lambda X)$. By positive homogeneity of $g$, we have $( \mathcal{E}_0(\lambda X)/ \lambda, Z/\lambda) \in \mathcal{A}( X )$, therefore $\mathcal{E}_0(\lambda X) \geq \lambda \mathcal{E}_0(X)$. 
	Using the same reasoning on $\mathcal{A}(X)$ we have $\mathcal{E}_0(\lambda X) \leq \lambda\mathcal{E}_0(X)$, hence $\mathcal{E}_0$ is positive homogeneous.

	The representation \eqref{E:coherent} follows from Theorem \ref{thm:rep-nonbounded} since the convex conjugate of the positive homogeneous function $g$ is the indicator of a closed convex set (i.e. it is either $0$ or $\infty$.)
\end{proof}
Let us conclude this section with an example.
\begin{example}
	Let $X$ be any random variable in $\mathcal{H}^1 $.
	Consider the BSDE
	\begin{equation}
	\label{eq:exampleBSDE}
		dY_t = - g(Y_t, Z_t),dt + Z_t\,dW_t, \quad Y_T = X
	\end{equation}
	with generator $g$ defined on $\mathbb{R} \times \mathbb{R}^d$ by
	\begin{equation*}
		g(y,z) := \begin{cases}
						 z^2/y  & \text{ if } y > 0, \, z \in \mathbb{R}^d\\
						 0               & \text{ if } y \le 0, \, z = 0\\
						 + \infty        & \text{ if } y \le 0, \, z \in \mathbb{R}^d \setminus \Set{0}. 
				  \end{cases}
	\end{equation*}
	The function $g$ satisfies the conditions of Theorem \ref{thm:rep-nonbounded}. 
	Therefore, the minimal supersolution $\mathcal{E}^{g}(X)$ of Equation \eqref{eq:exampleBSDE} admits the dual representation \eqref{eq:rep-nonbounded}. 
	Moreover, defining
	\begin{equation*}
		\mathcal{K} : = \Set{(\beta, q) \in \mathcal{D}_+ \times \mathcal{Q}: \beta \ge \frac{1}{4}\norm{q}^2 },
	\end{equation*} 
	one can check that $g^\ast$ takes the value $0$ on $\mathcal{K}$ and $+\infty$ on the complement of $\mathcal{K}$.
	Thus,
	\begin{equation*}
		\mathcal{E}^{g}_0(X) = \sup_{(\beta,q) \in \mathcal{K}}E_{Q^q}\bigl[ D^{\beta}_{0,T}X \bigr].
	\end{equation*}
\end{example}

%% file: chap05.tex
\section{Cash-Subadditive Risk Measures and BSDE}
	The operator $\mathcal{E}_0$ studied in the previous section can be seen as a risk measure.
	In fact, when the generator does not depend on $y$, the functional $\rho$ defined by $\rho(X) := \mathcal{E}_0(-X)$ is a convex risk measure in the sense of \citet{foellmer01}, and $u(X):= -\mathcal{E}_0(-X)$ defines a monetary utility function. 
	If the generator $g$ does depend on $y$ and satisfies \ref{dec}, then $\rho$ is instead a cash-subadditive risk measure as defined in \citep{ravanelli01}. 
	In particular, for all $m \geq 0$ holds $\rho( X - m ) \leq \rho( X ) + m $.

In this section we start with a cash-subadditive risk measure satisfying a given robust representation and show, in Theorem \ref{thm:rep-solution}, that such a risk measure must be the minimal supersolution of a BSDE.
Thus, we are given a dynamic cash-subadditive risk measure\footnote{Actually, this is only a risk measure up to a transformation as explained above.} of the form
\begin{equation}
	\label{risk_meas}
	\phi_{t}(X) := \esssup_{(\beta, q) \in \mathcal{D}_+\times\mathcal{Q}}E_{Q^q}\left[ D^\beta_{t,T} X - \int_{t}^TD^\beta_{t,u}f(\beta_u, q_u)\,du \Mid \mathcal{F}_{t} \right], \quad t\in [0,T],
\end{equation}
where $X$ is a random variable in $\mathcal{H}^1$ and $f: \mathbb{R}\times \mathbb{R}^d \rightarrow (-\infty, \infty ]$ a given proper function. 
A function $f$ is said to be 
\begin{enumerate}[label=(\textsc{Norm}),leftmargin=40pt]
	\item null at the origin if,  $f(0,0) = 0$.\label{norm}
\end{enumerate}

\begin{remark}
Since $D_{s,t}^\beta D_{t,u}^\beta=D_{s,u}^\beta$, the penalty function $\alpha$ defined by Equation \eqref{eq:rep-bounded-penalty} satisfies the following cocycle property introduced in \citep{Bion} for monetary convex risk measures:
\begin{equation}
\label{eq:cocycle}
	\alpha_{s,u}(\beta ,q) = \alpha_{s,t}(\beta,q) + E_{Q^q}\left[ D_{s,t}^\beta \alpha_{t,u}(\beta, q)\Mid \mathcal{F}_s \right]\quad \text{for every} \quad (\beta,q)\in \mathcal{D}_+ \times \mathcal{Q}.
\end{equation} 	
In the cash-additive case, the cocycle property takes the form
\begin{equation*}
	\alpha_{s,u}(q) = \alpha_{s,t}(q) + E_{Q^q}\left[\alpha_{t,u}(q)\mid \mathcal{F}_s \right].
\end{equation*}
Hence, 
the characterization of time-consistency in terms of the cocycle property given by \citep[Theorem 3.3]{Bion} shows that when $g$ does not depend on $y$, $\mathcal{E}$ is time-consistent even if the normalization condition $g(0) = 0$ is not assumed, compare \citep[Proposition 3.6]{DHK1101}. 
\end{remark}

In what follows we use the notation of the previous section. 
In particular, for any $q \in \mathcal{Q}$ we denote by $M^q$ the martingale density process of the probability measure $Q^q$ with respect to the reference measure $P$. 
We follow a method already put forth in \citet{BDH10} in the cash-additive case. The main idea is the following:
\begin{proposition}
	\label{Q-super}
 	For any $X \in \mathcal{H}^1$ and for each $(\beta,q) \in \mathcal{D}_+\times \mathcal{Q}$ the process 
	\begin{equation*}
		\varphi(X):=\left(D^\beta_{0,t}\phi_t(X) - \int_0^tD^\beta_{0,u}f(\beta_u,q_u)\,du\right)_{t\in[0,T]}
	\end{equation*}
 	is a $Q^q$-supermartingale.
\end{proposition}
\begin{proof}
	Let $0\leq s\leq t\leq T$. 
	We start by showing that the set 
	$$\Set{ E_{Q^{q}}\left[D^{\beta}_{t,T}X - \int_t^T D^{\beta}_{t,u} f(\beta_u,q_u)\,du \mmid \mathcal{F}_t \right]: (\beta,q) \in \mathcal{D}_+\times \mathcal{Q} }$$ is directed upward.
	Let $(\beta^1,q^1), (\beta^2, q^2) \in \mathcal{D}_+\times \mathcal{Q}$. 
	Let us define the stopping  time
	\begin{equation*}
		\tau := \inf\Set{ s > t: L^1_t < L^2_t  },
	\end{equation*}
	with $L^i_t := E_{Q^{q^i}}[D^{\beta^i}_{t,T}X - \int_t^T D^{\beta^i}_{t,u} f(\beta^i_u,q^i_u)\,du \mid \mathcal{F}_t ]$, $i = 1,2$, and put $\hat{q}: = q^11_{[0,\tau]} + q^21_{(\tau,T]} $ and $\hat{\beta} := \beta^11_{[0,\tau]} + \beta^21_{(\tau,T]} $.
	We have $(\hat{\beta}, \hat{q}) \in \mathcal{D}_+\times \mathcal{Q}$ and, by definition, $\hat{L}_t \ge \max\set{L^1_t, L^2_t}$, with $\hat{L}_t := E_{Q^{\hat{q}}}[D^{\hat{\beta}}_{t,T}X - \int_t^T D^{\hat{\beta}}_{t,u} f(\hat{\beta}_u,\hat{q}_u)\,du \mid \mathcal{F}_t ] $.
	
	Therefore, by \citep[Theorem A.32]{foellmer01}, there exists a sequence $(\beta^n, q^n) \subseteq \mathcal{D}_+\times\mathcal{Q}$ such that 
	\begin{equation*}
		\phi_t(X) = \lim_{n\rightarrow \infty } E_{Q^{q^n}}\left[D^{\beta^n}_{t,T}X - \int_t^T D^{\beta^n}_{t,u} f(\beta^n_u,q^n_u)\,du \mmid \mathcal{F}_t \right].
	\end{equation*}
	In addition, this convergence is monotone.
	Therefore, $\phi_t(X)$ is integrable, and it is also $ Q^q$-integrable for every $q\in \mathcal{Q}$ since $dQ^q/dP \in L^\infty$.
	Hence, for any $(\beta,q) \in \mathcal{D}_+\times \mathcal{Q}$, it holds
	\begin{multline*}
		 E_{Q^q}\left[\varphi_t(X)\Mid \mathcal{F}_s\right] = E_{Q^q}\left[ D^\beta_{0,t} \lim_{n\rightarrow \infty} E_{Q^{q^n}}\left[D^{\beta^n}_{t,T}X - \int_t^T D^{\beta^n}_{t,u} f(\beta^n_u,q^n_u)\,du \mmid \mathcal{F}_t \right]\mmid \mathcal{F}_s \right] \\
		  - E_{Q^q}\left[ \int_s^tD^\beta_{0,u}f(\beta_u, q_u)\,du \mmid \mathcal{F}_s \right]  
	 \shoveright{ - \int_0^s D^\beta_{0,u} f(\beta_u,q_u)\,du  } \\
	= \lim_{n\rightarrow \infty} D_{0,s}^\beta E_{Q^q}\left[E_{Q^{q^n}}\left[  D^\beta_{s,t}D^{\beta^n}_{t,T}X  - \int_t^T D^\beta_{s,t}D^{\beta^n}_{t,u}f(\beta_u^n,q_u^n)\,du -\int_s^t  D^\beta_{s,u} f(\beta_u,q_u)\,du\mmid \mathcal{F}_t \right] \mmid \mathcal{F}_s \right] \\
	 - \int_0^s D^\beta_{0,u} f(\beta_u,q_u)\,du,  
	\end{multline*}
    where the second equation follows by dominated convergence theorem.
    We put $\bar{\beta}^n = \beta1_{[0,t]} + \beta^n1_{(t,T]}$ and $\bar{q}^n = q1_{[0,t]} + q^n1_{(t,T]} $. 
	It follows that
	\begin{align*}
		 E_{Q^q}\left[\varphi_t(X)\Mid \mathcal{F}_s\right] &= D^\beta_{0,s}\lim_{n \rightarrow \infty}  E_{Q^{\bar{q}^n}}\left[ D^{\bar{\beta}^n}_{s,T} X - \int_s^T D^{\bar{\beta}^n}_{s,u} f(\bar{\beta}^n_u,\bar{q}^n_u)\,du \Mid \mathcal{F}_s \right]- \int_0^s D^\beta_{0,u} f(\beta_u,q_u)\,du\\
		   &\leq D^\beta_{0,s} \phi_s(X) - \int_0^s D^\beta_{0,u} f(\beta_u,q_u)\,du =\varphi_s(X),
	\end{align*}

where the inequality follows by definition of $\phi(X)$ and the fact that $(\bar{\beta}^n,\bar{q}^n) \in \mathcal{D}_+\times \mathcal{Q}$.
\end{proof}
Next we give two consequences of the previous result.
\begin{corollary}
\label{thm:q-martingale}
	Let $X \in \mathcal{H}^1$, suppose in addition that $\phi_0(X)$ admits a subgradient $(\beta, q)\in \mathcal{D}_+\times \mathcal{Q}$, i.e.  $(\beta, q)$ is such that $\phi_0(X) = E_{Q^q}\left[D^\beta_{0,T} X - \int_0^TD^\beta_{0,u} f(\beta_u,q_u)\,du \right]$. 
	Then for each $t\in [0,T]$ we have 
	\begin{equation}
	\label{t_subgrad}
		\phi_t(X) = E_{Q^q}\left[D^\beta_{t,T} X - \int_t^T D^\beta_{t,u} f(\beta_u,q_u)\,du \Mid \mathcal{F}_t \right],
	\end{equation}
	that is, $(\beta, q)$ is a subgradient of $\phi_t(X)$.
	Moreover, the process
	\begin{equation*}
		\left(D^\beta_{0,t} \phi_t(X) - \int_0^t D^\beta_{0,u} f(\beta_u,q_u)\,du\right)_{t\in[0,T]}
	\end{equation*}
	is a $Q^q$-martingale.
\end{corollary}
\begin{proof}
	Let $(\beta, q) \in \mathcal{D}_+\times\mathcal{Q}$ be such that 
	\begin{equation*}
		\phi_0(X) = E_{Q^q}\left[D^{\beta}_{0,T} X - \int_0^TD^\beta_{0,u} f(\beta_u,q_u)\,du \right].
	\end{equation*}
	By the previous proposition and the choice of $(\beta, q)$ we have for any $t \in [0,T]$  
	\begin{equation*}
		E_{Q^q}\left[D^\beta_{0,t} \phi_t(X) - \int_0^tD^\beta_{0,u} f(\beta_u,q_u)\,du \right]   \leq \phi_0(X) =  E_{Q^q}\left[ D^\beta_{0,T} X - \int_0^TD^\beta_{0,u} f(\beta_u,q_u)\,du \right],
	\end{equation*}
	from which ensues
	\begin{align*}
		E_{Q^q}\left[ D^\beta_{0,t} \phi_t(X) \right] & \leq E_{Q^q}\left[D^\beta_{0,T} X - \int_t^T D^\beta_{0,u} f(\beta_u,q_u)\,du \right]\\
 												  & = E_{Q^q}\left[ D^\beta_{0,t}  \left( D^\beta_{t,T} X - \int_t^T D^\beta_{t,u} f(\beta_u,q_u)\,du \right)\right].
	\end{align*}
	Since we have 
	\begin{equation*}
		\phi_t(X) \geq E_{Q^q}\left[ D^\beta_{t,T} X - \int_t^T D^\beta_{t,u} f(\beta_u,q_u)\,du \mmid \mathcal{F}_t\right],
	\end{equation*}
	and $0 < D^\beta_{0,t}  < \infty $ we conclude that 
	\begin{equation*}
		\phi_t(X) = E_{Q^q}\left[D^\beta_{t,T} X - \int_t^T D^\beta_{t,u} f(\beta_u,q_u)\,du \mmid \mathcal{F}_t \right] \quad \text{$Q^q$-a.s.}
	\end{equation*}

	From Equation \eqref{t_subgrad} we have, for all $t \in [0,T]$,
	\begin{equation*}
		D^\beta_{0,t} \phi_t(X) - \int_0^t D^\beta_{0,u} f(\beta_u,q_u)\,du
 		= E_{Q^q}\left[ D^\beta_{0,T} X - \int_0^T D^\beta_{0,u} f(\beta_u,q_u)\,du \mmid \mathcal{F}_t \right] \text{ $Q^q$-a.s.}
	\end{equation*}
\end{proof}
\begin{corollary}
	\label{decomp}
	Assume that the function $f$ satisfies \ref{norm}. 
	Then, for every $X \in \mathcal{H}^1$ the process $(\phi_t(X))_{t\in[0,T]}$ is a $P$-supermartingale and admits a Doob-Meyer decomposition of the form $\phi(X) = \phi_0(X) + M - A$ where $A$ is a c\`adl\`ag adapted and increasing process with $A_0 = 0$ and $M$ a continuous local martingale.
\end{corollary}
\begin{proof}
	The $P$-supermartingale property of $\phi(X)$ follows from Proposition \ref{Q-super} and the fact that $f(0,0) = 0$. Let us show that $\phi(X)$ has a c\`adl\`ag modification which is still a $P$-supermartingale. 
	Let $t \in [0,T]$, since $\phi(X)$ is a $P$-supermartingale, for all $s \in [t, T] \cap \mathbb{Q}$ we have $E[\phi_s(X) \mid \mathcal{F}_t] \leq \phi_t(X)$. 
	Hence, by Fatou's lemma and due to the fact that our filtration satisfies the usual conditions we obtain the inequality $\phi^+_t(X) \leq \phi_t(X)$, where
	\begin{equation*}
		\phi^+_t(X) := \lim_{s\downarrow t, s\in \mathbb{Q}}\phi_s(X).
	\end{equation*}
	On the other hand by continuity of martingales we have, for all $(\beta, q) \in \mathcal{D}_+\times\mathcal{Q} $,
	\begin{equation*}
		\phi^+_t(X) \geq  E_{Q^q}\left[ D^\beta_{t,T} X - \int_t^T D^\beta_{t,u} f(\beta_u,q_u)\,du \mmid \mathcal{F}_t \right] \text{ $P$-a.s.,}
	\end{equation*}
	so that taking the supremum with respect to $\beta, q$ yields $\phi^+_t(X) \geq \phi_t(X)$ $P$-a.s., thus we have $\phi^+(X) = \phi(X)$ $P$-a.s.
	We conclude by \citep[Proposition 1.3.14]{Karatzas1991} that  $\phi(X)$ has a c\`adl\`ag modification which is again a supermartingale. This path regularity of $\phi(X)$ ensures that it admits a Doob-Meyer decomposition.
\end{proof}

Now we want to link the dynamic risk measure defined by Equation \eqref{risk_meas} to a BSDE. 
In that regard, we assume that $f$ is \ref{jconv} and \ref{lsc}, and we denote by $g$ the function defined on $\mathbb{R}\times \mathbb{R}^d $ by
\begin{equation*}
g(y,z) := \sup_{\beta \ge 0; q \in \mathbb{R}^d}\Set{-\beta y + qz - f(\beta,q)}.
\end{equation*}
The function $g$ is \ref{dec} and if $f$ is \ref{norm} then $g$ is \ref{pos}. 
\begin{theorem}
	\label{thm:rep-solution}
	Assume that the function $f$ satisfies \ref{jconv}, \ref{lsc} and \ref{norm}.
	For all $X \in \mathcal{H}^1$, there exists a unique predictable $d$-dimensional process $Z$ such that $(\phi(X),Z)$ is the minimal supersolution of the BSDE with generator $g$ and terminal condition $X$.
\end{theorem}
\begin{proof}
	\emph{Supersolution property:} 
 		Let $X \in \mathcal{H}^1$. We start by proving that there exists $Z$ such that $(\phi(X), Z)$ is a supersolution of the BSDE with generator $g$ and terminal condition $X$. 
 		By Corollary \ref{decomp} there exist processes $A$ and $M$ such that $\phi_t(X) = \phi_0(X) + M_t - A_t$, and by martingale representation there exists a process $Z \in \mathcal{L}$ such that 
		\begin{equation}
		\label{decomp-rep}
			\phi_t(X) = \phi_0(X) + \int_0^tZ_u\,dW_u - A_t.
		\end{equation}
		By definition of $\phi(X)$ and Equation \eqref{decomp-rep}, $\int_0^tZ_u\,dW_u \ge E\left[X \mid\mathcal{F}_t \right] -\phi_0(X) $. 
		Thus, $\int Z\,dW$ is a supermartingale as a local martingale bounded from below by a martingale. 
		Let $(\beta, q) \in \mathcal{D}_+\times\mathcal{Q}$. 
		Applying It\^o's formula to $D^\beta_{0,t}\phi_t(X)$ leads us to
		\begin{align*}
			d\left(D^\beta_{0,t}\phi_t(X) \right) & = -\beta_tD^\beta_{0,t}\phi_t(X)\,dt + D^\beta_{0,t}d\phi_t(X)\\
									   			  & = -\beta_tD^\beta_{0,t}\phi_t(X)\,dt + D^\beta_{0,t}\left( -dA_t + Z_t\,dW_t \right)\\
 												  & = -\beta_tD^\beta_{0,t}\phi_t(X)\,dt + D^\beta_{0,t}\left(-dA_t + Z_tq_t\,dt \right) + D^\beta_{0,t} Z_t\,dW_t^{Q^q}.
		\end{align*}
		Therefore,
		\begin{multline}
		\label{q-b_decomp}
 			d\left(D^\beta_{0,t}\phi_t(X) - \int_0^tD^\beta_{0,u} f(\beta_u,q_u)\,du \right)\\
  			= D^\beta_{0,t} \left( -\beta_t\phi_t(X)\,dt  -dA_t + Z_tq_t\,dt - f(\beta_t,q_t)\,dt\right) + D^\beta_{0,t} Z_t\,dW_t^{Q^q}.
		\end{multline}
		By the $Q^q$-supermartingale property proved in Proposition \ref{Q-super}, we have
		\begin{equation*}
			dA_t \geq \left(-\beta_t\phi_t(X)  + q_tZ_t - f(\beta_t,q_t)\right)\,dt. 
		\end{equation*}
		Since $\beta$ and $q$ were taken arbitrary, it holds
		\begin{equation}
		\label{A-g}
			dA_t \geq g(\phi_t(X),Z_t)\,dt.
		\end{equation}
		Hence Equation \eqref{decomp-rep} gives, for all $0\leq s\leq t\leq T$,
		\begin{equation*}
			\phi_s(X) - \int_s^tg(\phi_u(X),Z_u)\,du + \int_s^tZ_u\,dW_u \geq \phi_t(X),
		\end{equation*}
		which shows that $(\phi(X), Z) $ is an admissible supersolution.
		
		\emph{Minimality:}
		Showing that the process $\phi(X)$ is minimal is done using exactly the same arguments as those used to prove Equation \eqref{eq:rep_h1_first-inequality} in the second step of the proof of Theorem \ref{thm:rep-nonbounded} and the first part of the proof of Theorem \ref{thm:rep-bounded1}. 
		Replacing $0$ by $t$ and the expectation by the conditional expectation in the proof of Equation \eqref{eq:rep_h1_first-inequality} does not affect the reasoning.
		Recalling that since $g$ is \ref{jconv}, \ref{dec} and \ref{pos} the minimal supersolution is unique concludes the proof.
\end{proof}
\begin{theorem}
	Assume that the function $f$ satisfies \ref{jconv}, \ref{lsc} and \ref{norm}. 
	Let $X \in \mathcal{H}^1$, if $\phi_0(X)$ admits a subgradient $(\beta, q) \in \mathcal{D}_+ \times \mathcal{Q}$ then the minimal supersolution $(\phi(X),Z)$ is actually a solution. 

	In addition, for $P\otimes dt$-almost all $(\omega,t) \in \Omega\times [0,T] $, $(\beta_t, q_t) \in \partial g(\omega,t,\phi_t(X),Z_t) $, subgradient of $g$ with respect to $(\phi_t(X), Z_t)$.
\end{theorem}
\begin{proof}
 	Let $X \in \mathcal{H}^1$ and $(\beta,q) \in \mathcal{D}_+\times \mathcal{Q}$ be a subgradient of $\phi_0(X)$. 
 	Then, by Corollary \ref{thm:q-martingale} and the decomposition appearing in Equation \eqref{q-b_decomp}, we have
	\begin{equation*}
		dA_t = \left(-\beta_t\phi_t(X) + q_tZ_t - f(\beta_t,q_t) \right)\,dt.
	\end{equation*}
	Definition of $g$ and Equation \eqref{A-g} give
	\begin{equation*}
		g(\phi_t(X),Z_t)\,dt \geq \left(-\beta_t\phi_t(X) + q_tZ_t - f(\beta_t,q_t) \right)\,dt = dA_t \geq g(\phi_t(X),Z_t)\,dt.
	\end{equation*}
	Then, $dA_t = g(\phi_t(X), Z_t)\,dt$, showing that $(\beta,q) \in \partial g(\phi(X), Z) $ $P\otimes dt$-a.s.
	Equation \eqref{decomp-rep} yields
	\begin{equation*}
		\phi_t(X) = X - \int_t^Tg(\phi_u(X), Z_u)\,du + \int_t^TZ_u\,dW_u.
	\end{equation*}
	Hence $(\phi(X),Z)$ is a solution.
\end{proof}
We conclude by the following complete characterization of the minimal supersolution suggested by Corollary \ref{thm:rep-bounded} and Theorem \ref{thm:rep-solution}.
\begin{theorem}
	Assume that the function $g$ satisfies \ref{jconv}, \ref{dec} and \ref{pos}, $g^\ast$satisfies \ref{norm}. 
	If $X \in L^{\infty}$, then the following are equivalent:
	\begin{itemize}
		\item[(i)] There exists a predictable $d$-dimensional process $Z$ such that $(\mathcal{E}(X), Z)$ is the minimal supersolution of the BSDE with terminal condition $X$ and driver $g$.
		\item[(ii)] The functional $\mathcal{E}$ admits the representation
		\begin{equation*}
			\mathcal{E}_{t}(X) = \esssup_{(\beta, q) \in \mathcal{D}_+\times\mathcal{Q}}E_{Q^q}\left[ D^\beta_{t,T} X - \int_{t}^T D^\beta_{t,u} g^\ast(\beta_u, q_u)\,du \mmid \mathcal{F}_{t} \right], \quad t \in [0,T].
		\end{equation*}
	\end{itemize}
\end{theorem}
	

%% file: appendix01.tex

\section{Some Properties of the Minimal Supersolution Operator}
\label{appendix01}
The aim of this appendix is to present the proofs of some properties of the minimal supersolution used in the paper.

\begin{proof}[Proof of Proposition \ref{thm:properties} ]
		See \citep[Proposition 3.2 and Theorems 4.9 and 4.12]{DHK1101}, but for the sake of readability we give the details for the points \ref{point3}, \ref{point4} and \ref{point5}.

	As for \ref{point3}, let $m \in \mathbb{R}$ with $m \geq 0$ and $X \in \mathcal{X}$. 
	Since $X + m \geq X$, if $\mathcal{A}(X) = \emptyset$ then $\mathcal{A}(X+m) = \emptyset$.
	In that case $\mathcal{E}_0(X+m) = \infty = \mathcal{E}_0(X)$.
	If $\mathcal{A}(X) \neq \emptyset$, let  $(Y,Z) \in \mathcal{A}(X)$.
	For all $0 \leq s\leq t\leq T$, since $g$ fulfills \ref{dec}, we have
	\begin{align*}
		Y_s + &m - \int_s^tg_u(Y_u+m,Z_u)\,du + \int_s^tZ_u\,dW_u \geq m + Y_s - \int_s^tg_u(Y_u,Z_u)\,du + \int_s^tZ_u\,dW_u \geq 	m + Y_t.
	\end{align*}
	Thus, $(Y+m,Z) \in \mathcal{A}(X+m)$, which implies $\mathcal{E}_0(X +m) \leq Y_0 +m$.
	Taking $Y = \mathcal{E}(X)$, we have $\mathcal{E}_0(X + m) \leq \mathcal{E}_0(X) + m$ showing the cash-subadditivity.
	
	As for \ref{point4}, if $g$ does not depend on $y$, one can show that $\mathcal{E}_0$ is additionally cash-superadditive, that is, $\mathcal{E}_0(X +m) \geq \mathcal{E}_0(X) + m$ for $m \geq 0$.
	Indeed, using the same argument we have $\mathcal{A}(X) \neq \emptyset$ implies $\mathcal{A}(X+m)\neq \emptyset$ and $(Y - m,Z) \in \mathcal{A}(X)$ for all $(Y,Z) \in \mathcal{A}(X+m)$.
	Then, if $g$ does not depend on $y$, it follows that $\mathcal{E}_0(X +m) = \mathcal{E}_0(X) +m$ for all $m\in \mathbb{R}_+$.
	Thus, $\mathcal{E}_0(X) + m =\mathcal{E}_0(X) + m^+ - m^- = \mathcal{E}_0(X + m^+) -m^- = \mathcal{E}_0(X + m+m^-) - m^- = \mathcal{E}_0(X + m)$ for all $m \in \mathbb{R}$. 

	As for \ref{point5}, if $g(y,0) = 0$, we have $(y,0) \in \mathcal{A}(y)$, and therefore $\mathcal{E}_0(y) \leq y$. If $g$ is \ref{pos}, for all $(Y,Z) \in \mathcal{A}(y)$, the supermartingale property of $Y$ and the terminal condition yield $Y_0 \geq E[Y_T ] \geq y$. Hence, $\mathcal{E}_0(y) \geq y$. 
\end{proof}

Next, we recall the proofs of the existence, uniqueness and monotone stability of the minimal supersolution with respect to the generator.
These results were already obtained in \citep{DHK1101}. 
Here we argue that their proofs are also valid, up to a slight change, if we replace the assumption \ref{dec} by \ref{jconv} on the generators.

Recall that for $X \in \mathcal{X}: = \set{X \in L^0: X^- \in L^1}$, the condition \ref{pos} ensures that the value process $Y$ of a supersolution $(Y,Z) \in \mathcal{A}(X)$ is a supermartingale such that 
\begin{equation}
\label{eq:supermartingalY}
	Y_t\geq - E[X^- \mid \mathcal{F}_t]\quad \text{for all } t \in [0,T],
 \end{equation} 
 see \citep[Lemma 3.3]{DHK1101}. 

\begin{proof}[Sketch of the Proof of Theorem \ref{thm:existence}]
	The uniqueness of $\bar{Z}$ follows by the supermartingale property of $\bar{Y}$ and the martingale representation theorem.
	The existence is proved by constructing, through concatenations, a sequence of supersolutions $(Y^n, Z^n)$ whose value processes $(Y^n)$ decrease to the process $\essinf\set{Y_t : (Y, Z) \in \mathcal{A}(X)}$. 
	By a compactness argument, a subsequence in the asymptotic convex hull of $(Z^n)$ which converges strongly to a process $\bar{Z}$ can be selected.
	The proof is completed by showing that there is a modification $\bar{Y}$ of $\essinf\set{Y_t : (Y, Z) \in \mathcal{A}(X)}$ such the candidate $(\bar{Y}, \bar{Z})$ is actually an admissible supersolution.
	In the case where $g$ does not satisfy \ref{dec} but \ref{jconv}, this is done as in the proof of the next theorem.
\end{proof}
  
\begin{theorem}
\label{thm:stability-withoutDEC}
	Let $X \in \mathcal{X} $ be a terminal condition, and let $(g^n)$ be an increasing sequence of generators, which converge pointwise to a generator $g$. 
	Suppose that each generator is defined on $\mathbb{R} \times \mathbb{R}^d$ and fulfills \ref{jconv}, \ref{lsc} and \ref{pos} and denote by $\bar{Y}^n $ the value process of the minimal supersolution of the BSDE with generator $g^n$.
	Then $\lim_{n\rightarrow \infty} \bar{Y}^n_0 = \mathcal{E}_0(X)$. 
	If, in addition, $\lim_{n \rightarrow \infty}\bar{Y}^n_0 < \infty $, then for all $t \in [0,T]$ the set $\mathcal{A}(X)$ is nonempty and $(\bar{Y}^n_t) $ converges $P$-a.s. to $\mathcal{E}_t(X)$.
\end{theorem}
\begin{proof}
	By monotonicity, see Proposition \ref{thm:properties}, the sequence $(\bar{Y}^n_0)$ is increasing. 
	Set $Y_0 = \lim_{n\rightarrow \infty}\bar{Y}^n_0$, if $Y_0 = \infty$ there is nothing to prove. 
	Else, we put $Y_t := \lim_{n}\bar{Y}^n_t$, $t \in [0,T]$. 
	It follows from the supermartingale property of $\bar{Y}^n$ and the monotone convergence theorem that $Y$ is a c\`adl\`ag supermartingale.
	Using the arguments of the proof of Theorem \ref{thm:existence}, we construct a candidate control $Z$ as pointwise limit of convex combinations $(\tilde{Z}^n )$ of  $(\bar{Z}^n)$, where $(\bar{Y}^n, \bar{Z}^n)$ is the minimal supersolution of the BSDE with generator $g^n$. 
	It remains to verify that $(Y,Z) \in \mathcal{A}(X) $. 
	Fatou's lemma gives
	\begin{equation*}
		Y_s - \int_s^tg_u(Y_u,Z_u)\,du + \int_s^tZ_u\,dW_u  \ge \limsup_{k\rightarrow \infty}\left( Y_s - \int_s^tg^k_u(Y_u,Z_u)\,du + \int_s^tZ_u\,dW_u \right).
	\end{equation*}
	And for every $k \le n$, denoting by $\lambda^n_i$ the convex weights of the convex combination $\tilde{Z}^n$, using \ref{jconv} we have
	\begin{align}
\nonumber		Y_s - \int_s^tg^k_u(Y_u,Z_u)\,du + \int_s^tZ_u\,dW_u & \ge \limsup_n\left( \tilde{Y}^n_s - \int_s^tg_u^k(\tilde{Y}^n_u, \tilde{Z}^n_u)\,du - \int_s^t\tilde{Z}^n_u\,dW_u \right)\\ 
\nonumber													 & \ge \limsup_n \sum_{i = n}^{M^n}\lambda^n_i \left( Y^i_s - \int_s^tg^k(Y^i_u, Z^i_u)\,du + \int_s^t Z^i_u\,dW_u \right)\\
\nonumber   												 & \ge \limsup_n \sum_{i = n}^{M^n}\lambda^n_i \left( Y^i_s - \int_s^tg^i(Y^i_u, Z^i_u)\,du + \int_s^t Z^i_u\,dW_u \right)\\
\label{eq:verif_solution}									 & \ge Y_t.
	\end{align}
	As to the admissibility of $Z$, by means of Equations \eqref{eq:supermartingalY} and \eqref{eq:verif_solution}, we have
	\begin{equation*}
		\int_0^tZ_u\,dW_u \ge -E[X^-\mid \mathcal{F}_t] - Y_0 
	\end{equation*}
	so that $\int Z\,dW$ is a supermartingale as a local martingale bounded from below by a martingale.
	Thus, $Z$ is admissible. 
\end{proof}